\makeatletter\let\ifGm@compatii\relax\makeatother

\documentclass[10pt]{amsart}
\usepackage{latexsym}
\usepackage{amssymb}
\usepackage{amscd}
\usepackage{epsf}
\usepackage{amsmath,amssymb,graphicx,amsthm,mathrsfs,verbatim}
\usepackage{pstricks}

\begin{document}

\theoremstyle{plain}
\newtheorem{Thm}{Theorem}[section]
\newtheorem{TitleThm}[Thm]{}
\newtheorem{Corollary}[Thm]{Corollary}
\newtheorem{Proposition}[Thm]{Proposition}
\newtheorem{Lemma}[Thm]{Lemma}
\newtheorem{Conjecture}[Thm]{Conjecture}
\theoremstyle{definition}
\newtheorem{Definition}[Thm]{Definition}
\theoremstyle{definition}
\newtheorem{Example}[Thm]{Example}
\newtheorem{TitleExample}[Thm]{}
\newtheorem{Remark}[Thm]{Remark}
\newtheorem{SimpRemark}{Remark}
\renewcommand{\theSimpRemark}{}

\numberwithin{equation}{section}

\newcommand{\C}{{\mathbb C}}
\newcommand{\Q}{{\mathbb Q}}
\newcommand{\R}{{\mathbb R}}
\newcommand{\Z}{{\mathbb Z}}
\newcommand{\mbS}{{\mathbb S}}
\newcommand{\mbU}{{\mathbb U}}
\newcommand{\mbO}{{\mathbb O}}
\newcommand{\mbG}{{\mathbb G}}
\newcommand{\mbH}{{\mathbb H}}

\newcommand{\flushpar}{\par \noindent}

\newcommand{\proj}{{\rm proj}}
\newcommand{\coker}{{\rm coker}\,}
\newcommand{\Sol}{{\rm Sol}}
\newcommand{\supp}{{\rm supp}\,}
\newcommand{\codim}{{\operatorname{codim}}}
\newcommand{\sing}{{\operatorname{sing}}}
\newcommand{\Tor}{{\operatorname{Tor}}}
\newcommand{\Hom}{{\operatorname{Hom}}}
\newcommand{\wt}{{\operatorname{wt}}}
\newcommand{\graph}{{\operatorname{graph}}}
\newcommand{\rk}{{\operatorname{rk}}}
\newcommand{\dlog}{{\operatorname{Derlog}}}
\newcommand{\Olog}[2]{\Omega^{#1}(\text{log}#2)}
\newcommand{\produnion}{\cup \negmedspace \negmedspace 
\negmedspace\negmedspace {\scriptstyle \times}}
\newcommand{\pd}[2]{\dfrac{\partial#1}{\partial#2}}

\def \ba {\mathbf {a}}
\def \bb {\mathbf {b}}
\def \bc {\mathbf {c}}
\def \bd {\mathbf {d}}
\def \bone {\boldsymbol {1}}
\def \bg {\mathbf {g}}
\def \bG {\mathbf {G}}
\def \bh {\mathbf {h}}
\def \bi {\mathbf {i}}
\def \bj {\mathbf {j}}
\def \bk {\mathbf {k}}
\def \bK {\mathbf {K}}
\def \bm {\mathbf {m}}
\def \bn {\mathbf {n}}
\def \bt {\mathbf {t}}
\def \bu {\mathbf {u}}
\def \bv {\mathbf {v}}
\def \by {\mathbf {y}}
\def \bV {\mathbf {V}}
\def \bx {\mathbf {x}}
\def \bw {\mathbf {w}}
\def \b1 {\mathbf {1}}
\def \bga {\boldsymbol \alpha}
\def \bgb {\boldsymbol \beta}
\def \bgg {\boldsymbol \gamma}

\def \itc {\text{\it c}}
\def \ite {\text{\it e}}
\def \ith {\text{\it h}}
\def \iti {\text{\it i}}
\def \itj {\text{\it j}}
\def \itm {\text{\it m}}
\def \itM {\text{\it M}} 
\def \itn {\text{\it n}}
\def \ithn {\text{\it hn}}
\def \itt {\text{\it t}}

\def \cA {\mathcal{A}}
\def \cB {\mathcal{B}}
\def \cC {\mathcal{C}}
\def \cD {\mathcal{D}}
\def \cE {\mathcal{E}}
\def \cF {\mathcal{F}}
\def \cG {\mathcal{G}}
\def \cH {\mathcal{H}}
\def \cK {\mathcal{K}}
\def \cL {\mathcal{L}}
\def \cM {\mathcal{M}}
\def \cN {\mathcal{N}}
\def \cO {\mathcal{O}}
\def \cP {\mathcal{P}}
\def \cS {\mathcal{S}}
\def \cT {\mathcal{T}}
\def \cU {\mathcal{U}}
\def \cV {\mathcal{V}}
\def \cW {\mathcal{W}}
\def \cX {\mathcal{X}}
\def \cY {\mathcal{Y}}
\def \cZ {\mathcal{Z}}

\def \ga {\alpha}
\def \gb {\beta}
\def \gg {\gamma}
\def \gd {\delta}
\def \ge {\epsilon}
\def \gevar {\varepsilon}
\def \gk {\kappa}
\def \gl {\lambda}
\def \gs {\sigma}
\def \gt {\tau}
\def \gw {\omega}
\def \gz {\zeta}
\def \gG {\Gamma}
\def \gD {\Delta}
\def \gL {\Lambda}
\def \gS {\Sigma}
\def \gW {\Omega}

\def \dim {{\rm dim}\,}
\def \mod {{\rm mod}\;}
\def \rank {{\rm rank}\,}

\newcommand{\ds}{\displaystyle}
\newcommand{\vf}{\vspace{\fill}}
\newcommand{\vect}[1]{{\bf{#1}}}
\def\R{\mathbb R}
\def\C{\mathbb C}
\def\CP{\mathbb{C}P}
\def\RP{\mathbb{R}P}
\def\N{\mathbb N}

\def\Sym{\mathrm{Sym}}
\def\Sk{\mathrm{Sk}}
\def\GL{\mathrm{GL}}
\def\Diff{\mathrm{Diff}}
\def\id{\mathrm{id}}
\def\Pf{\mathrm{Pf}}
\def\sll{\mathfrak{sl}}
\def\g{\mathfrak{g}}
\def\h{\mathfrak{h}}
\def\k{\mathfrak{k}}
\def\t{\mathfrak{t}}
\def\OcN{\mathscr{O}_{\C^N}}
\def\Ocn{\mathscr{O}_{\C^n}}
\def\Ocm{\mathscr{O}_{\C^m}}
\def\Ocnz{\mathscr{O}_{\C^n,0}}
\def\Derlog{\mathrm{Derlog}\,}
\def\expdeg{\mathrm{exp\,deg}\,}

\title[Characteristic Cohomology II: Matrix Singularities]
{Characteristic Cohomology II: Matrix Singularities}
\author[James Damon]{James Damon}

\address{Department of Mathematics, University of North Carolina, Chapel 
Hill, NC 27599-3250, USA
}



\begin{abstract}
For a germ of a variety $\cV, 0 \subset \C^N, 0$, a singularity $\cV_0$ of \lq\lq type $\cV$\rq\rq\, is given by a germ $f_0 : \C^n, 0 \to \C^N, 0$ which is transverse to $\cV$ in an appropriate sense so that $\cV_0 = f_0^{-1}(\cV)$.  In part I of this paper \cite{D6}  we introduced for such singularities the Characteristic Cohomology for the Milnor fiber (for 
$\cV$ a hypersurface), and complement and link (for the general case).  It captures the cohomology of $\cV_0$ inherited from $\cV$ and is given by subalgebras of the cohomology for $\cV_0$ for the Milnor fiber and complements and is a subgroup for the cohomology of the link.  We showed these cohomologies are functorial and invariant under groups of equivalences $\cK_{H}$ for Milnor fibers and $\cK_{\cV}$ for complements and links.  We also gave geometric criteria for detecting the non-vanishing of the characteristic cohomology.  
\par
 In this paper we apply these methods in the case $\cV$ denotes any of the varieties of singular $m \times m$ complex matrices which may be either general, symmetric or 
skew-symmetric (with $m$ even).  For these varieties we have shown in another paper that their Milnor fibers and complements have compact \lq\lq model submanifolds\rq\rq\, for their homotopy types, which are classical symmetric spaces in the sense of Cartan.  As a result, it follows that the characteristic cohomology subalgebras for the Milnor fibers and complements are images of exterior algebras (or in one case a module on two generators over an exterior algebra).  In addition, we extend these results to general $m \times p$ complex matrices in the case of the complement and link.  \par
We then apply the geometric detection method introduced in Part I to detect when characteristic cohomology for the Milnor fiber or complement contains a specific exterior subalgebra on $\ell$ generators and for the link that it contains an appropriate truncated and shifted version of the subalgebra.  The detection criterion involves a special type of \lq\lq kite map germ of size $\ell$\rq\rq based on a given flag of subspaces.  The general criterion which detects such nonvanishing characteristic cohomology is then given in terms of the defining germ $f_0$ containing such a kite map germ of size $\ell$.  Furthermore we use a restricted form of kite spaces to give a cohomological relation between the cohomology of local links and the global link.  
\end{abstract} 
\maketitle
\vspace{2ex}
\centerline{\bf Preliminary Version}
\vspace{1ex}
\par
\section*{Introduction}  
\label{S:sec0} 
\par
Let $\cV \subset M$ denote any of the varieties of singular $m \times m$ complex matrices which may be general, symmetric, or skew-symmetric ($m$ even), or $m  \times p$ matrices, in the corresponding space $M$ of such matrices.  A \lq\lq matrix singularity\rq\rq\, $\cV_0$ of \lq\lq type $\cV$\rq\rq\, for any of the $\cV \subset M$ is defined as $\cV_0 = f_0^{-1}(\cV)$ by a germ $f_0 : \C^n, 0 \to M, 0$ (which is transverse to $\cV$ in an appropriate sense).  In part I \cite{D6} we introduced the notion of characteristic cohomology for a singularity $\cV_0$ of type $\cV$ for any \lq\lq universal singularity\rq\rq  for the Milnor fiber (in case $\cV$ is a hypersurface) and for the complement and link (in the general case). In this paper we determine the characteristic cohomology for matrix singularities in all of these cases. 
\par
For matrix singularities the characteristic cohomology will give the analogue of characteristic classes for vector bundles (as e.g. \cite{MS}).  For comparison, a vector bundle $E \to X$ over CW complex $X$ is given by map $f_0 : X \to BG$ for $G$ the structure group of $E$ (e.g. $O_n$, $U_n$, $Sp_n$, $SO_n$, etc.).  It is well-defined up to isomorphism by the homotopy class of $f_0$.  Moreover the generators of $H^*(BG; R)$, for appropriate coefficient ring $R$ pull-back via $f_0^*$ to give the characteristic classes of $E$; so they generate a characteristic subalgebra of $H^*(X; R)$.  The nonvanishing of the characteristic classes which then give various properties of $E$.  Various polynomials in the classes correspond to Schubert cycles in the appropriate classifying spaces.  
\par
We will give analogous results for categories of matrix singularities of the various types.  Homotopy invariance is replaced by invariance under the actions of the groups of diffeomorphisms $\cK_H$ or $\cK_{\cV}$.  For these varieties we have shown in another paper \cite{D3} that they have compact \lq\lq model submanifolds\rq\rq\, for the homotopy types of both the Milnor fibers and the complements and these are classical symmetric spaces in the sense of Cartan.  As a result, it will follow that the characteristic subalgebra is the image of an exterior algebra (or in one case a module on two generators over an exterior algebra) on an explicit set of generators.  \par
We give a \lq\lq detection criterion\rq\rq\, for identifying in the characteristic sublgebra an exterior subalgebra on pull-backs of $\ell$ specific generators of the cohomology of the corresponding symmetric space.  It is detected by the defining germ $f_0$ containing a special type of \lq\lq unfurled kite map\rq\rq\, of size $\ell$.  This will be valid for the Milnor fiber, complement, and link.  
\par 
We will do this by using the support of appropriate exterior subalgebras of the Milnor fiber cohomology or of the complement cohomology for the varieties of singular matrices.  This is done using results of \cite{D4} giving the Schubert decomposition for the Milnor fiber and the complement to define \lq\lq vanishing compact models\rq\rq\, detecting these subalgebras.  In \S \ref{S:sec3} and \S \ref{S:sec8} we use the Schubert decompositions to exhibit vanishing compact models in the Milnor fibers and complements.  Then, we use the detection criterion introduced in Part I to give a criterion for detecting  nonvanishing exterior subalgebras of the characteristic cohomology using a class of \lq\lq unfurled kite maps\rq\rq.  Matrix singularities $\cV_0, 0$ defined by germs $f_0$ which contain such an \lq\lq unfurled kite map\rq\rq, are shown to have such subalgebras in their cohomology of Milnor fibers or complements and subgroups in their link cohomology.  In the case of general or skew-symmetric matrices, the results for the Milnor fibers and complements is valid for cohomology over $\Z$ (and hence any coefficient ring $R$); while for symmetric matrices, the results apply both for cohomology with coefficients in a field of characteristic zero or for $\Z/2\Z$-coefficients.  In all three cases for a field of characteristic zero, cohomology subgroups are detected for the links which are above the middle dimensions.  \par 
Furthermore, we extend in \S  \ref{S:sec8a} the results for complements and links for $m \times m$ matrices to general $m  \times p$ matrices.  This includes determining the form of the characteristic cohomology and giving a detection criterion using an appropriate form of kite spaces and mappings.  \par
A restricted form of the kite spaces serve a further purpose in \S \ref{S:sec8b} for identifying how the cohomology of local links of strata in the varieties of singular matrices relate to the cohomology of the global links. \par
\par 
\vspace{1ex}
\centerline{CONTENTS}
 \vspace{1ex}
\begin{enumerate}
\item	Matrix Equivalence for the Three Types of Matrix Singularities
\par 
\vspace{2ex} 
\par
\item	Cohomology of the Milnor Fibers of the $\cD_m^{(*)}$
\par 
\vspace{2ex} 
\par
\item	Kite Spaces of Matrices for Given Flag Structures 
\par
\vspace{2ex}
\par
\item	Detecting Characteristic Cohomology using Kite Maps of Matrices 
\par
\vspace{2ex}
\par
\item	Examples of Matrix Singularities Exhibiting Characteristic Cohomology
\par 
\vspace{2ex}
\par
\item	Characteristic Cohomology for the Complements and Links of Matrix Singularities
\par 
\vspace{2ex}
\par
\item	Characteristic Cohomology for Non-square Matrix Singularities
\par 
\vspace{2ex}
\par 
\item	Cohomological Relations between Local Links via Restricted Kite Spaces

\end{enumerate}
\vspace{3ex}
\section{Matrix Equivalence for the Three Types of Matrix Singularities} 
\label{S:sec2}
\par

We will apply the results in Part I \cite{D6} to the cohomology for a matrix singularity 
$\cV_0$ for any of the three types of matrices.  We let $M$ denote the space of $m \times m$ general matrices $M_m(\C)$, resp. symmetric matrices $Sym_m(\C)$, resp. skew-symmetric matrices $Sk_m(\C)$ (for $m$ even).  We also let $\cD_m^{(*)}$ denote the variety of singular matrices for each case with $(*)$ denoting $( )$ for general matrices, $(sy)$ for symmetric matrices, or $(sk)$ for skew-symmetric matrices.  For the case of $m \times p$ matrices with $m \not = p$ we use the notation $\cD{m, p} \subset M_{m, p}(\C)$. Also, the corresponding defining equations for the three $m \times m$ cases are given by: $\det$ for the general and symmetric cases and the Pfaffian $\Pf$ for the skew-symmetric case.  We generally denote the defining equation by $H : \C^N, 0 \to \C, 0$ for $\cV$, where $M \simeq \C^N$ for appropriate $N$ in each case and $\cV = \cD_m^{(*)}$.  For the case of $m \times p$ matrices with $m \not = p$, $\cD{m, p}$ is not a hypersurface  and we will not be concerned with its defining equation. 
\par
\subsection*{Abbreviated Notation for the Characteristic Cohomology} \hfill
\par 
For matrix singularities $\cV_0$ defined by $f_0 : \C^n \to M, 0 = \C^N, 0$, the characteristic cohomology with coefficients $R$ defined in Part I is denoted as follows: for the Milnor fiber (in the case $\cV$ is a hypersurface) by $\cA_{\cV}(f_0, R)$, for the complement of $\cV_0$ by $\cC_{\cV}(f_0, R)$, and for the link of $\cV_0$, for $\bk$ a field of characteristic $0$, by $\cB_{\cV}(f_0, \bk)$.  We use a simplified notation for matrix singularities.
\begin{align*}
\cA_m^{(*)}(f_0; R) \,\, = \,\, &\cA_{\cD_m^{(*)}}(f_0; R), \qquad  \cC^{(*)}(f_0; R) \,\, = \,\, \cC_{\cD_m^{(*)}}(f_0; R)\, , \\
 &\text{and} \quad \cB_m^{(*)}(f_0; \bk) \,\, = \,\, \cB_{\cD_m^{(*)}}(f_0; \bk)
\end{align*}  \par
If $m$ is understood, we shall suppress it in the notation. \par
For the case of $m \times p$ matrices with $m \not = p$, $\cD_{m, p}$ is not a hypersurface, but we shall use for the complement and link 
$$\cC_{m, p}(f_0; R) \,\, = \,\, \cC_{\cD_{m, p}}(f_0; R) 
 \quad \text{and} \quad \cB_{m, p}(f_0; \bk) \,\, = \,\, \cB_{\cD_{m, p}^{(*)}}(f_0; \bk)  $$
Again, if $(m, p)$ is understood, we shall suppress it in the notation.

\subsection*{Matrix Singularities Equivalences $\cK_M$ and $\cK_{HM}$} \hfill
\par
\par
There are several different equivalences that we shall consider for 
matrix singularities $f_0 : \C^n, 0 \to M, 0$ with $\cV$ denoting the 
subvariety of singular matrices in $M$.  The one used in classifications is 
$\cK_M$--{\it equivalence: }  We suppose that we are given an action of a 
group of matrices $G$ on $M$.  For symmetric or skew symmetric 
matrices, it is the action of $\GL_m(\C)$ by $B\cdot A =  B\, A\, B^T$.  For 
general $m \times p$ matrices, it is the action of $\GL_m(\C) \times 
\GL_p(\C)$ by $(B, C)\cdot A =  B\, A\, 
C^{-1}$.  Given such an action, then the group $\cK_M$ consists of pairs 
$(\varphi, B)$, with $\varphi$ a germ of a diffeomorphism of $\C^n, 0$ and 
$B$ a holomorphic germ $\C^n, 0 \to G, I$.  The action is given by
$$    f_0(x)  \mapsto  f_1(x)\,  = \, B(x)\cdot( f_0\circ \varphi^{-1}(x)) \, . $$
Although $\cK_M$ is a subgroup of $\cK_{\cV}$, they have the same tangent spaces and their path connected components of their orbits agree (for example this is explained in \cite[\S 2]{DP} because of the results due to J\'{o}zefiak \cite{J}, J\'{o}zefiak-Pragacz 
\cite{JP}, and Gulliksen-Neg\r{a}rd\cite{GN} as pointed out by Goryunov-Mond \cite{GM}).  \par 
We next restrict to codimension $1$ subgroups; let 
$$GL_m(\C)^{(2)} \, \, \overset{def}{=} \,\, \ker( \det \times \det: GL_m(\C) \times GL_m(\C) \to (\C^* \times \C^*)/\gD\C^*)$$ 
where $\gD\C^*$ is the diagonal subgroup.  We then replace the groups for 
$\cK_M$--equivalence by the subgroup $SL_m(\C)$ for the symmetric and 
skew-symmetric case and for the general case the subgroup $GL_m(\C)^{(2)}$.
These restricted versions of equivalence preserve the defining equation $H$ in each case.  We denote the resulting equivalence groups by $\cK_{HM}$, which are subgroups of the corresponding $\cK_{H}$.  As $\cK_{HM}$ equivalences preserve $H$, they also preserve the Milnor fibers and the varieties of singular matrices $\cV$.  
By the above referred to results, in each of the three cases, these $\cK_{HM}$ also have the same tangent spaces as $\cK_{H}$ in each case.  \par 
As a consequence of \cite[Prop. 2.1 and Prop. 2.2]{D6}, since $\cK_{HM}$ is a subgroup of $\cK_H$, we have for any coefficient ring $R$ and field $\bk$ of characteristic $0$ the following corollary.
\begin{Corollary}
\label{Cor2.3}
For each of the three cases of the varieties of $m \times m$ singular matrices $\cV =\cD_m^{(*)}$, let $\cV_0$ be defined by $f_0 : \C^n, 0 \to M, 0$, with $M$ denoting the corresponding space of matrices.  Then,\par 
\begin{itemize}
\item[a)]   the characteristic subalgebra $\cA^{(*)}(f_0; R)$ is, up to Milnor fiber cohomology isomorphism, an invariant of the $\cK_{HM}$--equivalence class of $f_0$; 
\item[b)]  $\cB^{(*)}(f_0; \bk)$ is, up an isomorphism of the cohomology of the link, an invariant of the $\cK_{M}$--equivalence class of $f_0$; and 
\item[c)] the characteristic subalgebra $\cC^{(*)}(f_0; R)$ is, up to an isomorphism of the cohomology of the complement, an invariant of the $\cK_{M}$--equivalence class of $f_0$.
\end{itemize}
\par 
Hence, the structure of the cohomology of the Milnor fiber of $\cV_0$ as a 
graded algebra (or graded module) over $\cA^{(*)}(f_0; R)$ is, up to isomorphism, independent of the $\cK_{HM}$--equivalence class of $f_0$
\end{Corollary}
\par
Before considering the cohomology of the Milnor fibers of the $\cD_m{(*)}$, we first give an important property which implies that each of the $\cD_m^{(*)}$ are $H$-holonomic in the sense of \cite{D2}, which gives a geometric condition that assists in proving that the matrix singularity is finitely $\cK_{HM}$-determined (and hence finitely 
$\cK_H$-determined).  This will be a consequence of the fact that for all three cases the above groups act transitively on the strata of the canonical Whitney stratification of 
$\cD_m^{(*)}$.    
\par
\begin{Lemma}
\label{Lem2.4}
For each of the three cases of $m \times m$ general, symmetric and skew-symmetric matrices, the corresponding subgroups $GL_m(\C)^{(2)}$ , resp. $SL_m(\C)$ act transitively on the strata of the canonical Whitney stratification of $\cD_m^{(*)} $.
\end{Lemma}
\par
\begin{proof}[Proof of Lemma \ref{Lem2.4}]
First, for the general case, let $A \in \cD_m$ have rank $r < m$.  We also denote the linear transformation on the space of column vectors defined by $A$ to be denoted by $L_A$.  Then, we let $\{\bv_1, \dots , \bv_m\}$ denote a basis for $\C^m$ so that 
$\{\bv_{r +1}, \dots , \bv_m\}$ is a basis for $\ker(L_A)$.  We also let $\{\bw_1, \dots , \bw_m\}$ denote a basis for $\C^m$ so that $\bw_j = L_A(\bv_j)$ for $j = 1, \dots, r$.  We let $b = \det(\bv_1 \dots  \bv_m)$ and $c = \det(\bw_1 \dots  \bw_m)$.  Then, we let $B^{-1} = (\bv_1, \dots , \bv_{m-1}, \frac{c}{b}\bv_m)$ and $C^{-1} = (\bw_1 \dots  \bw_m)$.  Then, $C\cdot A \cdot B^{-1} = \begin{pmatrix} I_r & 0  \\ 0 & 0 \end{pmatrix}$, where $I_r$ is the $r \times r$ identity matrix.  Also,$\det(B) = \det(C) = c$ so  $(B, C) \in GL_m(\C)^{(2)}$.  Thus, the each orbit of $GL_m(\C) \times GL_m(\C)$, which consists of matrices of given fixed rank $< m$ is a stratum of the canonical Whitney stratification, is also an orbit of $GL_m(\C)^{(2)}$.  \par
For both the symmetric and skew-symmetric cases the corresponding orbits under $GL_m(\C)$ consist of matrices of given symmetric or skew-symmetric type of fixed rank $< m$; and they form strata of the canonical Whitney stratification.  We show that they are also orbits under the action of $SL_m(\C)$.  \par
For $A \in \cD_m^{(sy)}$ of rank $r < m$, we consider the symmetric bilinear form 
$\psi(X, Y) = X^T\cdot A\cdot Y$ for column vectors in $\C^m$.  We can find a basis 
$\{\bv_1, \dots , \bv_m\}$ for $\C^m$ so that $\psi(\bv_i, \bv_i) =  1$ for $i = 1, \dots , r$, $= 0$ for $i > r$, and $\psi(\bv_i, \bv_j) =  0$ if $i \neq j$.  Then, let 
$b = \det(\bv_1 \dots  \bv_m)$,  we let 
$B^T = (\bv_1, \dots , \bv_{m-1}, \frac{1}{b}\bv_m)$.  Then $\det(B) = 1$ and 
$B \cdot A\cdot B^T = \begin{pmatrix} I_r & 0  \\ 0 & 0 \end{pmatrix}$. \par
Lastly, for the skew-symmetric case the argument is similar, except for $A \in \cD_m^{(sk)}$ of rank $r < m$, we consider the skew-symmetric bilinear form 
$\psi(X, Y) = X^T\cdot A\cdot Y$ for column vectors in $\C^m$ with even $m$ and $r= 2k$.  There is a basis $\{\bv_1, \dots , \bv_m\}$ for $\C^m$ so that $\psi(\bv_{2i-1}, \bv_{2i}) =  1$ for $i = 1, \dots , k$, and otherwise, $\psi(\bv_i, \bv_j) =  0$ for $i < j$.  Then, let $b = \det(\bv_1 \dots  \bv_m)$,  we let $B^T = (\bv_1, \dots , \bv_{m-1}, \frac{1}{b}\bv_m)$.  Then $\det(B) = 1$ and $B \cdot A\cdot B^T = \begin{pmatrix} J_k & 0  \\ 0 & 0 \end{pmatrix}$, where $J_k$ is the $r \times r$ block diagonal matrix with $k$ $2 \times 2$-blocks of $J_1 = \begin{pmatrix} 0 & 1  \\ -1 & 0 \end{pmatrix}$.
\end{proof} 
\par
\section{Cohomology of the Milnor Fibers of the $\cD_m^{(*)}$}
\label{S:sec3}
\par
We next recall results from \cite{D3} and \cite{D4} giving the cohomology structure 
of the Milnor fibers of the $\cD_m^{(*)}$ for each of the three types of matrices.  This includes: representing the Milnor fibers by global Milnor fibers, giving compact symmetric spaces as compact models for the homotopy types of the global Milnor fibers, giving the resulting cohomology for the symmetric spaces, geometrically representing the cohomology classes, and indicating the relation of the cohomology classes for different $m$. \par
\subsection*{Homotopy Type of Global Milnor fibers via Symmetric 
Spaces} \hfill 
\par 
The global Milnor fibers for each of the three cases, which we denote by , $F_m$, resp. $F_m^{(sy)}$, resp. $F_m^{(sk)}$, are given by $H^{-1}(1)$ for $H : M, 0 \to \C, 0$ the defining equation for $\cD_m^{(*)}$, which is $\det$ for the general of symmetric case and Pfaffian $\Pf$ for the skew-symmetric case.  As shown in \cite{D3} the Milnor fiber for the germ of $H$ at $0$ is diffeomorphic to the global Milnor fiber.   We reproduce here the representation of the global Milnor fibers as a homogeneous spaces, whose  homotopy types are given by symmetric spaces.  These provide compact models for the Milnor fibers diffeomorphic to their Cartan models as given by \cite[Table 1]{D4}.  these are given in Table \ref{Table1}.
\vspace{1ex}
\begin{table}[h]
\begin{tabular}{|l|c|c|c|l|}
\hline
Milnor   & Quotient  & Symmetric  & Compact Model & Cartan \\
 Fiber $F_m^{(*)}$  &  Space  &  Space  & $F_m^{(*)\, c}$ &  Model   \\
\hline
$F_m$  & $SL_m(\C)$ & $SU_m$  &  $SU_m$  & $F_m^{c}$ \\
\hline
$F_m^{(sy)}$  &  $SL_m(\C)/SO_m(\C)$  & $SU_m/SO_m$  &  $SU_m \cap 
Sym_{m}(\C)$  &  $F_m^{(sy)\, c}$ \\
\hline
$F_m^{(sk)}, m = 2n$ &  $SL_{2n}(\C)/Sp_{n}(\C)$  & $SU_{2n}/Sp_n$  &  
$SU_{m} \cap Sk_{m}(\C)$  & $F_{m}^{(sk)\, c}\cdot J_n^{-1}$  \\
\hline
\end{tabular}
\caption{Global Milnor fiber, its representation as a homogenenous space, 
compact model as a symmetric space, compact model as subspace and 
Cartan model.}
\label{Table1}
\end{table}
\par
\subsection*{Tower Structures of Global Milnor fibers and Symmetric 
Spaces by Inclusion} \hfill 
\par 
The global Milnor fibers for all cases, their symmetric spaces, and their compact models 
form towers via inclusions.  These are given as follows.  For the general and symmetric cases, there is the homomorphism $\tilde{\itj}_m : SL_m(\C) \hookrightarrow SL_{m+1} (\C)$ sending $A \mapsto \begin{pmatrix} A & 0 \\ 0 
& 1 \end{pmatrix}$.  This can be identified with the inclusion of Milnor fibers 
$\tilde{\itj}_m :F_m \subset F_{m+1}$.  Also, it restricts to give an inclusion 
$\tilde{\itj}_m : SU_m \hookrightarrow SU_{m+1}$ which are the compact models for the general case.  Second, it induces an inclusion $\tilde{\itj}^{(sy)}_m : SL_m(\C)/SO_m(\C) \hookrightarrow SL_{m+1} (\C)/SO_ {m+1}(\C)$ which is an inclusion of Milnor fibers $\tilde{\itj}^{(sy)}_m: F_m^{(sy)} \hookrightarrow F_{m+1}^{(sy)}$.  It also induces an inclusion of the compact homotopy models $\tilde{\itj}^{(sy)}_m : SU_m/SO_m(\R) \subset SU_{m+1}/SO_ {m+1}(\R)$ for the Milnor fibers.
\par 
For the skew symmetric case, the situation is slightly more subtle. 
First, the composition of two of the above successive inclusion homomorphisms for $SL_m(\C)$ gives a homomorphism $SL_m(\C) \hookrightarrow SL_{m+2} (\C)$ sending $A \mapsto \begin{pmatrix} A & 0 \\ 0 & I_2 \end{pmatrix}$ for the $2 \times 2$ identity matrix $I_2$.  For even $m = 2k$, it induces an inclusion $\tilde{\itj}^{(sk)}_m: SL_m(\C)/Sp_k(\C) \hookrightarrow SL_{m+2} (\C)/Sp_ {k+1}(\C)$.  However, the inclusion of Milnor fibers $ \tilde{\itj}^{(sk)}_m : F_m^{(sk)} \hookrightarrow F_{m+2}^{(sk)}$ is given by the map sending
 $A \mapsto \begin{pmatrix} A & 0 \\ 0 & J_1 \end{pmatrix}$ for the $2 \times 2$ 
skew-symmetric matrix $J_1 = \begin{pmatrix} 0 & 1 \\ -1 & 0 \end{pmatrix}$.  These two inclusions are related via the action of $SL_m(\C)/Sp_k(\C)$ on $F_m^{(sk)}$ which induces a diffeomorphism given by $B  \mapsto B\cdot J_k \cdot B^T$ ($m = 2k$).  
This also induces an inclusion of compact homotopy models  
$SU_m \cap Sk_m(\C) \subset SU_{m+2} \cap Sk_ {m+2}(\C)$.  This inclusion commutes with both the inclusion of the Milnor fibers under the diffeomorphism given in \cite{D3} by the action, and the inclusion of the Cartan models induced from the compact models after multiplying by $J_k^{-1}$, see Table \ref{Table1}.  
The Schubert decompositions for all three cases given in \cite{D4} satisfy the additional property that they respect the inclusions.  \par 
\subsection*{Cohomology of Global Milnor fibers using Symmetric Spaces} \hfill 
\par 
Next, we recall the form of the cohomology algebras for the global Milnor fibers. 
First, for the $m \times m$ matrices for the general case or skew-symmetric case (with $m = 2n$), by Theorems \cite[Thm. 6.1]{D4} and \cite[Thm. 6.14]{D4}, the Milnor fiber cohomology with coefficients $R = \Z$ are given as follows. 
\begin{align}
\label{Eqn4.2}
  H^*(F_m; \Z) \,\,  &\simeq \,\, 
\gL^*\Z \langle e_3, e_5, \dots , e_{2m-1} \rangle \,  \quad \text{general case }  \\
H^*(F_m^{(sk)}; \Z) \,\,  &\simeq \,\, \gL^*\Z \langle e_5, e_9, \dots , e_{4n-3} \rangle \quad \text{skew-symmetric case ($m = 2n$)} \, .
\end{align}
Therefore these isomorphisms continue to hold with $\Z$ replaced by any coefficient ring $R$.  Thus, for any coefficient ring $R$, $\cA^{(*)}(f_0; R)$ is the quotient ring of a 
free exterior $R$-algebra on generators $e_{2j-1}$, for $j = 2, 3, \dots , m$, 
resp. $e_{4j-3}$ for $j = 2, 3, \dots , n$.  
\par
For the $m \times m$ symmetric case there are two important cases where 
either $R = \Z/2\Z$ or $R = \bk$, a field of characteristic zero.  
First, for the coefficient ring $R = \bk$, the 
symmetric case breaks-up into two cases depending on whether $m$ is even 
or odd (see \cite[Thm. 6.7 (2), Chap. 3]{MT} or Table 1 of \cite{D3}).  
\begin{equation}
\label{Eqn4.3}
  H^*(F_m^{(sy)}; \bk) \,\,  \simeq \,\, \begin{cases}
\gL^*\bk \langle e_5, e_9, \dots , e_{2m-1} \rangle \,  & \text{ if $m = 
2k+1$ }  \\
\gL^*\bk \langle e_5, e_9, \dots , e_{2m-3} \rangle \{1, e_m\}& \text{if 
$m = 2k$ } \, .
\end{cases}
\end{equation}
Here $e_m$ is the Euler class of a $m$-dimensional real oriented vector bundle 
$\tilde{E}_m$ on the Milnor fiber $F_m^{(sy)}$. The vector bundle $\tilde{E}_m$ on the symmetric space $SU_m/SO_m(\R)$ has the form 
$SU_m \times_{SO_m(\R)} \R^m \to SU_m/SO_m(\R)$ where the action of $SO_m(\R)$ is given by the standard representation.  This can be described as the {\em bundle of totally real subspaces} of $\C^m$, which is the bundle of 
$m$-dimensional real subspaces of $\C^m$ whose complexifications are $\C^m$.  \par
In the second case for $R = \Z/2\Z$, by 
Theorem \cite[Thm. 6.15]{D4} using \cite[Thm. 6.7 (3), Chap. 3]{MT},  we have
\begin{equation}
\label{Eqn4.4}
  H^*(F_m^{(sy)}; \Z/2\Z) \,\,  \simeq \,\, 
\gL^*\Z/2\Z \langle e_2, e_3, \dots ,e_m \rangle \,   
\end{equation}
for generators $e_j = w_j(\tilde{E}_m)$, for $j = 2, 3, \dots , m$, 
for $w_j(\tilde{E}_m)$ the $j$-th Stiefel-Whitney class of the real oriented 
$m$-dimensional vector bundle $\tilde{E}_m$ above.  
\par
We summarize the structure of the characteristic subalgebra $\cA^{(*)}(f_0; R)$ in each case with the following.
\begin{Thm}
\label{Thm4.5}
Let $f_0 : \C^n, 0 \to M, 0$ define a matrix singularity $\cV_0, 0$
for $M$ the space of $m \times m$ matrices which are either general, symmetric, or 
skew-symmetric (with $m = 2n$). 
\begin{itemize}
\item[i)]  In the general and skew-symmetric cases, $\cA^{(*)}(f_0; R)$ is a quotient of the free $R$-exterior algebra with generators given in \eqref{Eqn4.2}
\item[ii)]  In the symmetric case with $R = \Z/2\Z$, $\cA^{(sy)}(f_0; \Z/2\Z)$ 
is the quotient of the free exterior algebra over $\Z/2\Z$ on generators $e_j = w_j(\tilde{E}_m)$, for $j = 2, 3, \dots , m$, for $w_j(\tilde{E}_m)$ the Stiefel-Whitney classes of the real oriented $m$-dimensional vector bundle $\tilde{E}_m$ on the Milnor fiber of $\cD_m^{(sy)}$.  Hence, $\cA^{(*)}(f_0; \Z/2\Z)$ is a subalgebra generated by the Stiefel-Whitney classes of the pull-back vector bundle $f_{0, w}^*(\tilde{E}_m)$ on 
$\cV_w$. 
\item[iii)]  In the symmetric case with $R = \bk$, a field of characteristic 
zero, $\cA^{(sy)}(f_0; \bk)$ is a quotient of the $\bk$-algebras in each of the cases 
in \eqref{Eqn4.3}.
\end{itemize}
Then, in each of these cases, the cohomology (with coefficients in a ring $R$) of 
the Milnor fiber of $\cV_0$ has a graded module structure over the characteristic 
subalgebra $\cA^{(*)}(f_0; R)$ of $f_0$.
\end{Thm}
\par
\subsection*{Cohomology Relations Under Inclusions for Varying $m$} \hfill
\par 
We give the relations between the cohomology of the global Milnor fibers and the symmetric spaces for varying $m$ under the induced inclusion mappings.  
The relations are the following.
\begin{Proposition}
\label{Prop4.5}
\begin{itemize}
\item[1)]  In the {\em general case}, for the inclusions $\tilde{\itj}_{m-1} : SU_{m-1} \hookrightarrow SU_{m}$ and $\tilde{\itj}_{m-1} : F_{m-1} \subset F_m$, 
$\tilde{\itj}_{m-1}^*$ is an isomorphism on the subalgebra generated by $\{ e_{2i-1} : i = 2, \dots , m-1\}$ and  $\tilde{\itj}_{m-1}^*(e_{2m-1}) = 0$. 
\item[2)] In the {\em skew-symmetric case (with $m = 2n$)}, for the inclusions 
$\itj^{(sk)}_{m-2} : SU_{2(n-1)}/Sp_{n-1} \hookrightarrow SU_{2n}/Sp_n$ and for Milnor fibers $\tilde{\itj}^{(sk)}_{m-2} : F_{m-2}^{(sk)} \hookrightarrow F_{m}^{(sk)}$, 
$\tilde{\itj} ^{(sk)\, *}_{m-2}$ is an isomorphism on the subalgebra generated by 
$\{ e_{4i-3} : i = 2, \dots , m-1\}$ and  $\tilde{\itj} ^{(sk)\, *}_{m-2}(e_{4m-3}) = 0$.
\item[3)] In the {\em symmetric case}, for the inclusion $\tilde{\itj}^{(sy)}_{m-1} : SU_{m-1}/SO_{m-1}(\R) \hookrightarrow SU_{m} /SO_{m}(\R)$ and for Milnor fibers  
$\tilde{\itj}^{(sy)}_{m-1}: F_{m-1}^{(sy)} \subset F_m^{(sy)}$: \par
\begin{itemize}
\item[i)]  for coefficients $R = \Z/2\Z$, $\tilde{\itj}^{(sy)\, *}_{m-1}$ is an isomorphism on the subalgebra generated by $\{ e_i : i = 2, \dots , m-1\}$ and  
$\tilde{\itj}^{(sy)\, *}_{m-1}(e_m) = 0$;
\item[iia)]  for coefficients $R = \bk$, a field of characteristic $0$, if $m = 2k$, then 
$\tilde{\itj}^{(sy)\, *}_{m-1}$ is an isomorphism on the subalgebra generated by $\{ e_{4i-3} : i = 2, \dots , k\}$, and $\tilde{\itj}^{(sy)\, *}_{m-1}(e_m) = 0$, and 
\item[iib)] if $m = 2k+1$, then $\tilde{\itj}^{(sy)\, *}_{m-1}$ is an isomorphism on the subalgebra generated by 
$\{ e_{4i-3} : i = 2, \dots , k\}$, and $\tilde{\itj}^{(sy)\, *}_{m-1}(e_{2m-1}) = 0$.
\end{itemize}
\end{itemize}
\end{Proposition}
\begin{proof}
For the general and skew-symmetric cases, the Schubert decomposition for the Cartan models $\cC_m$ and $\cC_m^{(sk)}$ for successive $m$ given in \cite{D4} preserves the inclusions and the homology properties.  In these two cases the result follows from the resulting identified Kronecker dual cohomology classes \cite[\S 6]{D4}.  \par
For the symmetric case and for $\Z/2\Z$-coefficients, an analogous Schubert decomposition gives the corresponding result.  The remaining symmetric case for coefficients $\bk$ a field of characteristic $0$ does not follow in \cite{D4} from the Schubert decomposition.  Instead, the computation of the cohomology of the symmetric space given in \cite[Chap. 3]{MT} yields the result.  In fact the algebraic computations in \cite [Chap. 3]{MT} (see also \cite{Bo}) also give the results for the other cases.  
\end{proof}
 \par
\subsection*{Vanishing Compact Models for the Milnor Fibers of $\cD_m^{(*)}$} \hfill
\par 
In part I we gave a detection criterion \cite[Lemma 3.2]{D6} for detecting the nonvanishing of a subgroup $E$ of the characteristic cohomology of the Milnor fiber 
$\cA^{(*)}(f_0, R)$.  We did so using vanishing compact models for the Milnor fiber.  We  use the preceding compact models for the Milnor fibers to give vanishing compact models for detecting nonvanishing subalgebras of $\cA^{(*)}(f_0, R)$.  From the above, let $F_M^{(*), c}$ denote the compact models for the individual global Milnor fibers $F_M^{(*)}$. 
We define $\Phi : F_m^{(*), c} \times (0, 1] \to H^{-1}((0, 1])$ sending 
$\Phi(A, t)= t\cdot A$.  Also, let $E = \gL^* R\{e_{i_1}, \dots, e_{i_\ell}\}$ denote the exterior subalgebra of $H^*(F_m^{(*), c}; R)$ on generators of the $\ell$ lowest degrees. We also let $\gl_E : F_{\ell}^{(*), c} \to F_m^{(*), c}$ denote the compositions 
$\tilde{\itj}^{(*)}_{m-1}\circ \cdots \circ \tilde{\itj}^{(*)}_{\ell}$.  Then, by Proposition \ref{Prop4.5}, $\gl_E^*$ induces an isomorphism from $E$ to its image.  Our goal is to first show that an appropriate restriction of $\Phi$ to a subinterval $(0, \gd)$ will provide a vanishing compact model for $F_M^{(*)}$; and moreover, we will use $\gl_E^*$ to give a germ which detects $E$.  
First, we give vanishing compact models for each case as follows.
\begin{Proposition}
\label{Prop4.6}   
A vanishing compact model for the Milnor fiber for $\cD_M^{(*)}$ is given for sufficiently small $0 < \gd << \gevar$ by $\Phi : F_m^{(*), c} \times (0, \gd) \to H^{-1}((0, \gevar])$ sending 
$\Phi(A, t)= t\cdot A$.  
\end{Proposition}
\begin{proof}
We begin by first making a few observations about the global Milnor fibers.  For $M$ one of the spaces of $m \times m$ matrices, we consider $H : M, 0 \to \C, 0$ the defining equation for $\cD_m^{(*)}$ ($H = \det$ or $\Pf$).  Then, the global Milnor fiber is $F_m^{(*)} = H^{-1}(1)$.  Now we can consider multiplication in $M$ by a constant $a \neq 0$.  As $H$ is homogeneous, if $A \in F_m^{(*)}$, then $a\cdot A \in H^{-1}(a^m)$ in the general or symmetric cases, or in the skew-symmetric cases 
$H^{-1}(a^k)$ where $m = 2k$.  \par
We also observe that multiplication by $a$ is a diffeomorphism between these two Milnor fibers.  We denote the image of $F_m^{(*)}$ by multiplication by $a$ by $aF_m^{(*)}$.  Then, by e.g. the proof of \cite[Lemma 1.2]{D3}, given $\gd > 0$, there is an $a > 0$ so that $aF_m^{(*)} \cap B_{\gd}$ is the local Milnor fiber of $\cV_0$, $aF_m^{(*)}$ is transverse to  the spheres of radii $\geq \gd$, and $aF_m^{(*)} \cap B_{\gd} \subset aF_m^{(*)}$ is a homotopy equivalence.  \par
Also, we have the compact homotopy models which occur as submanifolds of $SU_m$ of the form $SU_m$ for the general case, resp. $SU_m \cap Sym_m(\C)$ for the symmetric case, resp. $SU_m \cap Sk_m(\C)$  for the skew-symmetric case.  Now, for the standard Euclidean norm on $M_n(\C)$, $\| A\| = \sqrt{m}$ for $A \in SU_m$.  Then, as well this holds for $SU_m \cap Sym_m(\C)$, and for $SU_m \cap Sk_m(\C)$.  We denote the compact model in $F_m^{(*)}$ by $F_m^{(*)\, c}$.  Then, in each case if $M \simeq \C^N$, $F_m^{(*)\, c} \subset S^{2N-1}_{\sqrt{m}}$, the sphere of radius 
$\sqrt{m}$.  Thus, $aF_m^{(*)\, c} \subset S^{2N-1}_{a\sqrt{m}}$.   \par
Then, we first choose $0 < \eta << \gd < 1$ so that  $H : H^{-1}(B^*_{\eta}) \cap B_{\gd} \to B^*_{\eta}$ is the Milnor fibration of $H$.  \par 
We choose $0 < a < \eta$ so that also $a\sqrt{m} < \gd$.  Then, we observe the composition $aF_m^{(*)\, c} \subset aF_m^{(*)}\cap B_{\gd} \subset aF_m^{(*)}$ is 
a homotopy equivalence.  Hence, The restriction $\Phi : F_m^{(*)\, c} \times (0, a) \to H^{-1}(B^*_a) \cap B_{\gd} \to B^*_a$ restricts to a homotopy equivalence for each $0 < t < a$ and so gives a vanishing compact model.   
\end{proof}
\vspace{1ex} 
In light of Theorem \ref{Thm4.5}, there are several natural problems to be solved involving the characteristic cohomology for matrix singularities of each of the types.
\flushpar
{\it Problems for the Characteristic Cohomology of the Milnor Fibers of Matrix Singularities: } \par
\begin{itemize}
\item[1)]  Determine the characteristic subalgebras as the images of the exterior algebras by detecting which monomials map to nonzero elements in 
$H^*(\cV_w ; R)$.  
\item[2)] Identify geometrically these non-zero monomials in 1) via the 
pull-backs of the Schubert classes. 
\item[3)]  For the symmetric case with $\Z/2\Z$-coefficients, compute the 
Stiefel-Whitney classes of the pull-back of the vector bundle $\tilde{E}_m$.
\item[4)] Determine a set of module generators for the cohomology of 
the Milnor fibers as modules over the characteristic subalgebras.
\end{itemize}
We will give partial answers to these problems in the next sections.
\par
\section{Kite Spaces of Matrices for Given Flag Structures} 
\label{S:sec4}
We begin by introducing for a flag of subspaces for $\C^m$, a {\em linear kite subspace} of size $k$ in the space of $m \times m$ matrices of any of the three types: general $M_m(\C)$, symmetric $Sym_m(\C)$, or skew-symmetric $Sk_m(\C)$ (with $m$ even).  We initially consider the standard flag for $\C^m$, given by $0 \subset \C \subset \C^2 \subset \cdots \subset \C^{m-1} \subset \C^{m}$.  We choose coordinates 
$\{x_1, \cdots , x_m\}$ for $\C^m$ so that $\{x_1, \cdots , x_k\}$ are coordinates for 
$\C^k$ for each $k$.  \par 
We let $E_{i, j}$ denote the $m \times m$ matrix with entry $1$ in the $(i, j)$-position and $0$ otherwise. We also let $E^{(sy)}_{i, j} = E_{i, j} + E_{j, i}$, $i < j$, or $E^{(sy)}_{i, i} = E_{i, i}$ for the space of symmetric matrices; and $E^{(sk)}_{i, j} = E_{i, j} - E_{j, i}$, for $i < j$.   Then, we define
\begin{Definition}
\label{Def4.1}  For each of the three types of $m \times m$ matrices and the standard flag of subspaces of $\C^m$, the corresponding {\em linear kite subspace of size $\ell$} is the linear subspace of the space of matrices defined as follows:
\begin{itemize}
\item[i)] For $M_m(\C)$, it is the linear subspace $\bK_m(\ell)$ spanned by \par
$$  \{E_{i, j} : 1 \leq i, j \leq \ell\} \cup \{E_{i, i} : \ell < i \leq m\}$$
\item[ii)] For $Sym_m(\C)$, it is the linear subspace $\bK_m^{(sy)}(\ell)$ spanned by \par
$$  \{E^{(sy)}_{i, j} : 1 \leq i \leq j \leq \ell\} \cup \{E_{i, i}^{(sy)} : \ell < i \leq m\}$$
\item[iii)] For $Sk_m(\C)$ with $m$ even, for $\ell$ also even, it is the linear 
subspace $\bK_m^{(sk)}(\ell)$ spanned by \par
$$  \{E^{(sk)}_{i, j} : 1 \leq i < j \leq \ell\} \cup \{E^{(sk)}_{2i, 2i+1} : \ell < 2i < m\}$$
\end{itemize}
\par
Furthermore, we refer to the germ of the inclusion $\iti_m^{(*)}(\ell) : \bK_m^{(*)}(\ell), 0 \to M, 0$, for each of the three cases as a {\em linear kite map of size $\ell$}.
\end{Definition}
The general form of elements \lq\lq the kites\rq\rq\, in the linear kite subspaces have the form given in \eqref{Eqn4.1}.  
\par
\begin{equation}
\label{Eqn4.1}
Q_{\ell, m - \ell} \,\, = \,\, \begin{pmatrix} A_{\ell} & 0_{\ell, m - \ell}  \\
		0_{m - \ell, \ell} & D_{m - \ell} 
\end{pmatrix}
\end{equation}
where $A_{\ell}$ is an $\ell \times \ell$-matrix which denotes an arbitrary matrix in either $M_{\ell}(\C)$, resp. $Sym_{\ell}(\C)$, resp. $Sk_{\ell}(\C)$; and $0_{r, s}$ denotes the zero $r \times s$ matrix.  Also, $D_{m- \ell}$ denotes an arbitrary 
$(m- \ell) \times (m- \ell)$ diagonal matrix in the general or symmetric case as in Figure \ref{fig:altkitefig}.  In the skew symmetric case, $D_{m- \ell}$ denotes the $2 \times 2$ block diagonal matrix with skew-symmetric blocks of the form given by \eqref{Eqn4.1b} as in Figure \ref{fig:altskewkitefig}, with 
\begin{equation*}
\label{Eqn4.1b}
J_1(*) \,\, = \,\, \begin{pmatrix} 0 & *  \\
		-* & 0 
\end{pmatrix}
\end{equation*}
and \lq\lq\, * \rq\rq\, denoting an arbitrary entry.
\par
\vspace{1ex} 
\begin{figure}
$$ 
\begin{pmatrix} * & \cdots & * & 0  & \cdots & 0 \\
	\cdots & \cdots & \cdots & 0  & \cdots & 0 \\
 *& \cdots & * & 0  & \cdots & 0 \\
0 & \cdots & 0 & *  & \cdots & 0 \\
0 & \cdots & 0 & 0 & \ddots  & 0 \\
0 & \cdots & 0 & 0 & \cdots  & *
\end{pmatrix}
$$
\caption{Illustrating the form of elements of a linear kite space of size $\ell$ in either the space of general matrices or symmetric matrices. For general matrices the upper left matrix of size $\ell \times \ell$ is a general matrix, while for symmetric matrices it is symmetric.}
\label{fig:altkitefig}
\end{figure}
\par
\vspace{1ex} 
\begin{figure}
$$ 
\begin{pmatrix} * & \cdots & * & 0  & \cdots & 0 \\
	\cdots & \cdots & \cdots & 0  & \cdots & 0 \\
 *& \cdots & * & 0  & \cdots & 0 \\
0 & \cdots & 0 & J_1(*)  & \cdots & 0 \\
0 & \cdots & 0 & 0 & \ddots  & 0 \\
0 & \cdots & 0 & 0 & \cdots  & J_1(*)
\end{pmatrix}
$$
\caption{Illustrating the form of elements of a linear \lq\lq skew-symmetric kite\rq\rq space of size $\ell$ (with $\ell$ even) in the space of skew-symmetric matrices. The upper left $\ell \times \ell$ matrix is a skew-symmetric matrix.}
\label{fig:altskewkitefig}
\end{figure}
\par
We next extend this to general flags, and then to nonlinear subspaces as follows. 
For each of the three types of matrices $M =$ $M_m(\C)$, resp. $Sym_m(\C)$, resp. $Sk_m(\C)$ (with $m$ even). 

\begin{Definition}
\label{Def3.3} 
An {\em unfurled kite map} of given matrix type is any element of the orbit of 
$\iti_m^{(*)}(\ell)$, for $(*) = ( )$, resp. $(sy)$, resp.$(sk)$, under the corresponding equivalence group $\cK_{HM}$.   \par
A germ $f_0 : \C^n, 0 \to M, 0$ {\em contains a kite map of size $\ell$} for each of the three cases if there is a germ of an embedding $g : \bK_m^{(*)}(\ell), 0 \to \C^n, 0$ such that $f_0 \circ g$ is an unfurled kite map. 
\end{Definition}
\par 
\begin{Remark}
\label{Rem3.4} 
We note that unfurled kite maps have the property that the standard flag can be replaced by a general flag; and moreover, the flag and linear kite space can undergo nonlinear deformations.  These can be performed by iteratively applying appropriate row and column operations using elements of the local ring of germs on $\C^n,0$ instead of constants.  
\end{Remark}
A simple example of an unfurled kite map is given in Figure \ref{fig:unfkitefig}.
\par
\vspace{5ex} 
\begin{figure}
$$ 
\begin{pmatrix} 
x_{1,1} + 2 x_{4,4} x_{1,3} & x_{1,2} + 2 x_{4,4} x_{2,3}   & x_{1,3} + 2 x_{4,4} x_{3,3} & 0   \\
x_{1,2} + 2 x_{4,4} x_{2,3} & x_{2,2} -35 x_{1,2} x_{4,4} & x_{2,3} & (7 x_{1,2}- 5) x_{4,4}   \\
x_{1,3} + 2 x_{4,4} x_{3,3} & x_{2,3} & x_{3,3} & 0  \\
0 & (7 x_{1,2}- 5) x_{4,4} & 0 & x_{4,4}  
\end{pmatrix}
$$
\caption{An example of an unfurled kite map of size $3$ into $4 \times 4$ symmetric matrices.}
\label{fig:unfkitefig}
\end{figure}
\par
\section{Detecting Characteristic Cohomology using Kite Spaces of Matrices} 
\label{S:sec5}
\par
In \S 3 of Part I, we gave a detection criterion \cite[Lemma 3.2]{D6} for detecting the nonvanishing of a subgroup $E$ of the characteristic cohomology of the Milnor fiber 
$\cA^{(*)}(f_0, R)$.  In this section we use this criterion using kite maps to detect nonvanishing exterior subalgebras of $\cA^{(*)}(f_0, R)$.  In \S \ref{S:sec3}, we gave in equations \eqref{Eqn4.2}, \eqref{Eqn4.3}, and \eqref{Eqn4.4}  the cohomology of the Milnor fibers for the $\cD_m^{(*)}$ for each of the three types of matrices.  Thus, as for any matrix singularity $f_0 : \C^n, 0 \to M, 0$, by Theorem \ref{Thm4.5} the characteristic subalgebra is a quotient of the corresponding algebra.  
We let $E = \gL^* R\{e_{i_1}, \dots, e_{i_\ell}\} \subseteq H^*(F_m^{(*), c}; R)$ denote an exterior algebra on generators of the $\ell$ lowest degrees.  Then, using the map $\gl_E$ given before Proposition \ref{Prop4.6}, $\gl_E^*$ induces an isomorphism from $E$ to its image.  We next use $\gl_E$ to show that for germs $f_0$ containing a kite map of size $\ell$ for each case detects $E$ in 
$\cA^{(*)}(f_0, R)$.  
\begin{Thm}
\label{Thm5.1}
Let $f_0 : \C^n, 0 \to M, 0$ define an $m \times m$ matrix singularity of one of the three types.  \par
\begin{itemize}
\item[a)]  In the case of general matrices, if $f_0$ contains an unfurled kite map of size 
$\ell < m$, then $\cA(f_0, R)$ contains an exterior algebra of the form 
$$\gL^*R \langle e_3, e_5, \dots , e_{2\ell-1} \rangle\, .$$
 on $\ell - 1$ generators.
\item[b)]  In the case of skew-symmetric matrices (with $m$ even), if $f_0$ contains an unfurled skew-symmetric kite map of size $\ell (= 2k) < m$, then $\cA^{(sk)}(f_0, R)$ contains an exterior algebra of the form 
$$ \gL^*R \langle e_5, e_9, \dots , e_{4k-3} \rangle \, .$$
 on $k-1$ generators.  
\item[c)]  In the case of symmetric matrices, if $f_0$ contains an unfurled symmetric kite map of size $\ell < m$, then $\cA^{(sy)}(f_0, R)$ contains an exterior algebra of one of the forms 
\begin{align*}
&\gL^*\bk \langle e_3, e_5, \dots , e_{2\ell-1} \rangle  \qquad \text{ if $R = \bk$ is a field of characteristic $0$, }  \\
&\gL^*\Z/2\Z \langle e_2, e_3, \dots , e_{\ell} \rangle  \qquad \text{ if $R = \Z/2\Z$ , } 
\end{align*}
\end{itemize}
\end{Thm}
\begin{Remark}
\label{Rem5.2}
In the symmetric case, it follows from c) that if $f_0$ contains an unfurled symmetric kite map of size $\ell < m$, then the Stiefel-Whitney classes of the pull-back bundle $w_i(f_{0, w}^*(\tilde{E}_m))$ on $\cV_w$ are non-vanishing for $i = 2, \dots , \ell$. 
\end{Remark}
\begin{proof}
By Theorem \ref{Thm4.5} and the Detection Lemma \cite[Lemma 3.2]{D6}, it is sufficient to show that the corresponding kite maps of each type detect the corresponding exterior subalgebra.  We use the notation from the proof of Proposition \ref{Prop4.6} which gave the vanishing compact models for the Milnor fibers in each case. \par
Then, we choose $0 < \eta << \gevar < \gd < 1$ so that  $H : H^{-1}(B^*_{\eta}) \cap B_{\gd} \to B^*_{\eta}$ is the Milnor fibration of $H$ and $H\circ \iti_m^{(*)}(\ell) : (H\circ \iti_m^{(*)}(\ell))^{-1}(B^*_{\eta}) \cap B_{\gevar} \to B^*_{\eta}$ is the Milnor fibration of $H\circ \iti_m^{(*)}(\ell)$.  
We also choose $0 < a < \eta$ so that $a\sqrt{m} < \gevar$.  

\par
Then there are the following inclusions.
\begin{equation}
\label{Eqn5.3}
aF_{\ell}^{(*)\, c} \,\, \subset\,\,  \iti_m^{(*)}(\ell)(H\circ \iti_m^{(*)}(\ell))^{-1}(a^r) \cap B_{\gevar}) \,\, \subset \,\, aF_m^{(*)}\cap B_{\gevar} \,\, \subset \,\, aF_m^{(*)}\, , 
\end{equation} 
 where $r = m$ in the general or symmetric case or $r = \frac{m}{2}$ in the 
skew-symmetric case.  The composition of inclusions 
$F_{\ell}^{(*)\, c} \subset F_m^{(*)\, c} \subset F_m^{(*)}$ commutes with multiplication by $a$ as in \eqref{CD5.5} where each vertical map is a diffeomorphism given by multiplication by $a$.   
\begin{equation}
\label{CD5.5}
\begin{CD}
   {F_{\ell}^{(*)\, c}} @>>> {F_m^{(*)\, c}} @>>> {F_m^{(*)}}\\
  @VVV @VVV @VVV \\
{aF_ {\ell} ^{(*)\, c}} @>{\iti^{(*)}_m(\ell)}>> {aF_m^{(*)\, c}}  @>>> {aF_m^{(*)}}\\
\end{CD} 
\end{equation}
Also, $\iti^{(*)}_m(\ell)$ in the bottom row is given by the map in \eqref{Eqn5.4}.
\begin{equation}
\label{Eqn5.4}
aA \,\, \mapsto \,\, aQ_{\ell, m - \ell} \,\, = \,\, \begin{pmatrix} aA_{\ell} & 0_{\ell, m - \ell}  \\
		0_{m - \ell, \ell} & aD_{m - \ell} 
\end{pmatrix}
\end{equation} 
\par
Then, by Proposition \ref{Prop4.5} the induced homomorphisms in cohomology for the top row of \eqref{CD5.5} restrict to an isomorphism on the corresponding exterior subalgebra of $H^*(F_m^{(*)} ; R)$ onto the cohomology $H^*({F_{\ell}^{(*)\, c}}; R)$, and vanishing on the remaining generators.  Hence, as the vertical diffeomorphisms induce isomorphisms on cohomology, the induced homomorphisms on cohomology for the bottom row have the same property.  Lastly, in \eqref{Eqn5.3}, the induced homomorphisms in cohomology restrict to an isomorphism on the corresponding exterior subalgebra of $H^*({F_{m}^{(*)\, c}}; R)$ to $H^*({aF_{\ell}^{(*)\, c}}; R)$.  Thus the induced homomorphism to the Milnor fiber of $H\circ \iti_m^{(*)}(\ell)$,  
$$ H^*({aF_{m}^{(*)\, c}}; R) \longrightarrow H^*(H\circ \iti_m^{(*)}(\ell))^{-1}(a^r) \cap B_{\gevar}; R)$$
restricts to an isomorphism of the corresponding exterior algebra onto its image.  Thus, 
the cohomology of the Milnor fiber of $H\circ \iti_m^{(*)}(\ell)$ contains the claimed exterior subalgebra.  Thus, the flag map $\iti^{(*)}_m(\ell)$ detects the corresponding exterior algebra, so the result follows by the Detection Lemma. 
\end{proof}
\section{Examples of Matrix Singularities Exhibiting Characteristic Cohomology} 
\label{S:sec6}
\par
We consider several examples illustrating Theorem \ref{Thm5.1}. \par 
\begin{figure}[h]
$$ 
\begin{pmatrix} 
x_{1,1} + 2 x_{4,4} x_{1,3} & x_{1,2} + 2 x_{4,4} x_{2,3}   & x_{1,3} + 2 x_{4,4} x_{3,3} & y_1   \\
x_{1,2} + 2 x_{4,4} x_{2,3} & x_{2,2} -35 x_{1,2} x_{4,4} & x_{2,3} +y_1 x_{1,1}^2 & (7 x_{1,2}- 5) x_{4,4}   \\
x_{1,3} + 2 x_{4,4} x_{3,3}  & x_{2,3} +y_1 x_{1,1}^2  & x_{3,3} + y_2 x_{2,2}^2  & y_2 \\
y_1 & (7 x_{1,2}- 5) x_{4,4} & y_2 & x_{4,4}  
\end{pmatrix}
$$
\caption{An example of a germ $f_0$ containing an unfurled kite map of size $3$ into $4 \times 4$ symmetric matrices in Figure \ref{fig:unfkitefig}.}
\label{fig:unfkitefigex1}
\end{figure}
\par
\begin{Example}
\label{Exam6.1}
Let $f_0 ; \C^9, 0 \to Sym_4(\C), 0$ be defined by $f_0(\bx, \by)$ given by the matrix in Figure \ref{fig:unfkitefigex1} for $\bx = (x_{1,1}, x_{1,2}, x_{1,3}, x_{2,2}, x_{2,3}, x_{3,3}, x_{4,4})$ and $\by = (y_1, y_2)$.  We let 
$\cV_0 = f_0^{-1}(\cD_4^{(sy)})$.  This is given by the determinant of the matrix in Figure \ref{fig:unfkitefigex1} defining $f_0$.  Then, $\cV_0$ has singularities in codimension $2$.
We observe that when $\by = (0, 0)$ we obtain the unfurled kite map in Figure \ref{fig:unfkitefig}.  Thus, by Theorem \ref{Thm5.1}, the Milnor fiber of $\cV_0$ has cohomology with $\Z/2\Z$ coefficients containing the subalgebra $\gL^*\Z/2\Z \langle e_2, e_3 \rangle$, so there is $\Z/2\Z$ cohomology in degrees $2$, $3$, and $5$.  We also note that $e_j = w_j(f_{0, w}^*\tilde{E}_4)$ so that one consequence is that the second and third Stiefel-Whitney classes of the pullback of the vector bundle 
$\tilde{E}_4$ are non-zero.  \par
 For coefficients a field $\bk$ of characteristic $0$, the cohomology of the Milnor fiber of 
$\cV_0$ has an exterior algebra $\gL^*\bk \langle e_5 \rangle$, so there is a $\bk$ generator $e_5$ in degree $5$.  \par
By Kato-Matsumota \cite{KM}, as singularities have codimension $2$, the Milnor fiber is simply connected.  Then, we can use the preceding to deduce information about the integral cohomology of the Milnor fiber from the universal coefficient theorem.  It must have rank at least $1$ in dimension $5$, and it has $2$-torsion in dimension $2$.  
\end{Example}
Second, we consider a general matrix singularity.
\begin{figure}[h]
$$ 
\begin{pmatrix} 
x_{1,1}  & x_{1,2}   & x_{1,3} & x_{1,4}  & g_1(\bx, \by)   \\
x_{2,1}  & x_{2,2} & x_{2,3} & x_{2,4} &  g_2(\bx, \by) \\
x_{3,1} & x_{3,2}  & x_{3,3}  & x_{3,4}  & g_3(\bx, \by) \\
x_{4,1} & x_{4,2}  & x_{4,3}  & x_{4,4}  & g_4(\bx, \by) \\
y_1 & y_2 & y_3 & y_4 &  x_{5,5}  
\end{pmatrix}
$$
\caption{An example of a germ $f_0$ in Example \ref{Exam6.2}, containing a linear kite map of size $4$ into $5 \times 5$ general matrices with $g_i(\bx, 0) \equiv 0$ for each
 $i$).}
\label{fig:unfkitefigex2}
\end{figure}
\par
\begin{Example}
\label{Exam6.2}
 We let $f_0 ; \C^{21}, 0 \to M_5(\C), 0$ be defined with 
$f_0(\bx, \by)$ given by the matrix in Figure \ref{fig:unfkitefigex2} for $\bx = (x_{1,1}, \dots , x_{4,4}, x_{5,5})$ and $\by = (y_1, y_2, y_3, y_4)$.  In this example we require that $g_i(\bx, 0) \equiv 0$ for each $i$ . We let $\cV_0 = f_0^{-1}(\cD_5)$.  This is given by the determinant of the matrix in Figure \ref{fig:unfkitefigex2} defining $f_0$.  Then, the $\cV_0$ has singularities in codimension $4$ in $\C^{21}$; hence by Kato-Matsumoto, the Milnor fiber is $2$-connected.
We observe that when $\by = (0, 0, 0, 0)$ we obtain the linear kite map $\iti_5(4)$.
Thus, by Theorem \ref{Thm5.1}, the Milnor fiber of $\cV_0$ has characteristic cohomology with integer coefficients containing the subalgebra $\gL^*\Z \langle e_3, e_5, e_7 \rangle$.  Hence, the integer cohomology has rank at least $1$ in dimensions $0, 3, 5, 7, 8, 10, 12, 15$.  We cannot determine at this point whether the generator $e_9$ maps to a nonzero element in the cohomology of the Milnor fiber of $\cV_0$.  Even if it does, there are several products involving $e_9$ in exterior algebra for the cohomology of $\cD_5$ must map to $0$, as the Milnor fiber is homotopy equivalent to a CW-complex of dimension $20$.  
\end{Example}
\subsubsection*{Module structure over $\cA^{(*)}(f_0, R)$ for the Cohomology of the Milnor fiber} 
\par
\begin{Remark}
\label{Rem7.7} \par
In \S 4 of part I, we considered how in the hypersurface case the cohomology of the Milnor fiber is a module over the characteristic cohomology and listed four issues which must be addressed.  Already for condition i) and $\cV = \cD_m^{(*)}$, this leaves the remaining issues to be addressed:
\begin{itemize}
\item[1)] giving a sufficient condition that guarantees that the partial criterion \cite[(4.2)]{D6} is satisfied to ensure that for the singular Milnor number $\mu_{\cV}(f_0)$ there is a contribution of a summand of that rank..
\item[2)] determining $\mu_{\cV}(f_0)$ for $\cV = \cD_m^{(*)}$.  In the case that 
$\cV_0$ has an isolated singularity (which requires that $n$ is small, i.e.
$n \leq \codim(\sing(\cD_m^{(*)}))$, but allows arbitrary $m$), Goryunov-Mond \cite{GM} give a formula in all three cases for $\mu_{\cV}(f_0)$ in terms of the formula of \cite{DM} for free divisors with a correction term given by an Euler characteristic for a Tor sequence.  Alternatively, by a different method using \lq\lq free completions\rq\rq in all three cases, with arbitrary $n$ but for small $m$, Damon-Pike \cite{DP} give formulas for $\mu_{\cV}(f_0)$ as alternating sums of lengths of explicit determinantal modules.  However, there still does not exist a formula valid for all $m$ and $n$.  
\end{itemize}
\end{Remark}
\section{Characteristic Cohomology for the Complements and Links of Matrix Singularities} 
\label{S:sec8}
\par
We now turn to the characteristic cohomology of the complement and link for matrix singularities of all three types.  Again, we may apply the Second Detection Lemma of Part I \cite[Lemma 3.4]{D6} for complements to detect a nonvanishing subalgebra of 
$\cC^{(*)}(f_0, R)$ and corresponding nonvanishing subgroups of $\cB^{(*)}(f_0, \bk)$.
In order to apply the earlier results to the cases of matrix singularities, we first recall in Table \ref{Coh.Compl.Lnk} the cohomology, with coefficients a field $\bk$ of characteristic $0$, of the complements and links as given in \cite[table 2]{D3}.  
We will then use the presence of kite maps to detect both subalgbras of $\cC^{(*)}(f_0, R)$ for the complements and subgroups of $\cB^{(*)}(f_0, \bk)$ for the links. 
\begin{Thm}
\label{Thm8.7}
Let $f_0 : \C^n, 0 \to M, 0$ define a matrix singularity $\cV_0$ of any of the three types.  If $f_0$ contains a kite map of size $\ell$, then the characteristic cohomology of the complement $\cC^{(*)}(f_0, \bk)$, for a field $\bk$ of characteristic $0$, contains an exterior algebra given by Table \ref{V_0.compl.link}.  \par 
Furthermore, the characteristic cohomology of the link $\cB^{(*)}(f_0, \bk)$, as a 
graded vector space contains the graded subspace given by truncating the exterior subalgebra of $\cC^{(*)}(f_0, \bk)$ listed in column $2$ of Table \ref{V_0.compl.link} in the top degree and shifting by the amount listed in the last column.  \par 
For the complements in the general and skew-symmetric cases, $\bk$ may be replaced by any coefficient ring $R$.  
\end{Thm}
\par
\begin{table}
\begin{tabular}{|l|c|c|l|}
\hline
Determinantal  & Complement  & $H^*(M \backslash \cD,\bk) \simeq$   & 
\,\,Shift  \\
  Hypersurface &  $M \backslash \cD$   &  $H^*(K/L,\bk)$   &   \\
\hline
$\cD_m^{sy}$ & $GL_m(\C)/O_m(\C)$ &  $\gL^*\bk\langle e_1, e_5, \dots , 
e_{2m-1}\rangle$  & $\binom{m+1}{2} - 2$ \\  
(m = 2k+1)  & $\sim U_m/O_m(\R)$    &    &     
 \\
\hline
$\cD_m^{sy}$ & $GL_m(\C)/O_m(\C)$ & $\gL^*\bk\langle e_1, e_5, \dots , 
e_{2m-3}\rangle$  &  $\binom{m+1}{2} + m - 2$ \\
(m = 2k)  &     &     & 
\\
\hline
$\cD_m$  & $GL_m(\C) \sim U_m$ & $\gL^*\bk\langle e_1, e_3, \dots , 
e_{2m-1}\rangle$  &  $m^2 - 2$  \\
\hline
$\cD_m^{sk}$ &  $GL_{2k}(\C)/Sp_{k}(\C)$  & $\gL^*\bk\langle e_1, e_5, 
\dots , e_{2m-3}\rangle$  &  $\binom{m}{2} - 2$  \\
(m = 2k) &  $\sim U_{2k}/Sp_{k}$   &     &    \\
\hline
\end{tabular}
\caption{The cohomology of the complements $M \backslash \cD$ and links 
$L(\cD)$ for each determinantal hypersurface $\cD$.  The complements, 
are homotopy equivalent to the quotients of maximal compact subgroups 
$K/L$ with cohomology given in the third column, where the generators of 
the cohomology $e_k$ are in degree $k$; and the structure is an exterior 
algebra.  For the links $L(\cD)$, the cohomology is isomorphic as a vector 
space to the cohomology of the complement truncated in the top degree 
and shifted by the degree indicated in the last column.}
\label{Coh.Compl.Lnk}
\end{table}
\par
\begin{Remark}
\label{Rem8.5}
In what follows to simplify statements, instead of referring to the complement of $\cV_0, 0 \subset \C^n, 0$ as $B_{\gevar} \backslash \cV_0$ for sufficiently small $\gevar > 0$, we will just refer to the complement as $\C^n \backslash \cV_0$, with the understanding that it is restricted to a sufficiently small ball.  
\end{Remark}
\par  
\begin{proof}[Proof of Theorem \ref{Thm8.7}]
The proof is similar to that for Theorem \ref{Thm5.1}.  As the statements are independent of $f_0$ in a given $\cK_{\cV}$-equivalence class, we may apply an element of $\cK_H$ to obtain an $f_0$ containing a linear kite map.  It is sufficient to show, as for the case of Milnor fibers, that the linear kite map detects the indicated subalgebra in $\cC^{(*)}(f_0, \bk)$, and then apply Alexander duality for the result for the link. \par
By the results in \cite{D3} summarized in Table \ref{Coh.Compl.Lnk}, the complement $M \backslash \cD_m^{(*)}$ is given by a homogeneous space $G/H$ which has as a compact homotopy model $(K/L)$ where $K = U_m$ for each of the cases.  For successive values of $m$, we have for the three cases the successive inclusions:  
\begin{itemize}
\item[i)] for the general case, $GL_m(\C) \hookrightarrow GL_{m+1} (\C)$ by $A \mapsto \begin{pmatrix} A & 0 \\ 0 & 1 \end{pmatrix}$; 
\item[ii)] for the symmetric case, $GL_m(\C) \hookrightarrow GL_{m+1} (\C)$ sending $A \mapsto \begin{pmatrix} A & 0 \\ 0 & 1 \end{pmatrix}$ induces an inclusion $GL_m(\C)/O_m(\C) \hookrightarrow GL_{m+1}(\C) /O_{m+1}(\C)$;  
\item[iii)] for even $m = 2k$, $GL_m(\C) \hookrightarrow GL_{m+2} (\C)$ sending $A \mapsto \begin{pmatrix} A & 0 \\ 0 & I_2 \end{pmatrix}$, for $I_2$ the $2 \times 2$ identity matrix, induces an inclusion $GL_m(\C)/Sp_k(\C) \hookrightarrow GL_{m+2} (\C)/Sp_ {k+1}(\C)$.
\end{itemize}
Then, these are obtained by the action of $GL_m(\C)$ on the appropriate spaces of matrices.  They restrict to the compact homogenenous spaces which are homotopy equivalent models for the complements, given in Table \ref{Coh.Compl.Lnk} and which we denote by $K/L$ for each of the three cases.  Also, the inclusions correspond to the following inclusions of spaces of matrices. 
\begin{itemize}
\item[i)] for the general case, $M_m(\C) \hookrightarrow M_{m+1} (\C)$ by $A \mapsto \begin{pmatrix} A & 0 \\ 0 & 1 \end{pmatrix}$; 
\item[ii)] for the symmetric case, $Sym_m(\C) \hookrightarrow Sym_{m+1} (\C)$ sending $A \mapsto \begin{pmatrix} A & 0 \\ 0 & 1 \end{pmatrix}$;  
\item[iii)] for even $m = 2k$, $Sk_m(\C) \hookrightarrow Sk_{m+2} (\C)$ sending $A \mapsto \begin{pmatrix} A & 0 \\ 0 & J_1 \end{pmatrix}$, for the $2 \times 2$ 
skew-symmetric matrix $J_1 = \begin{pmatrix} 0 & 1 \\ -1 & 0 \end{pmatrix}$.
\end{itemize}
\par
Furthermore, for the cohomology of these spaces (via their homotopy equivalent 
compact models $K/L$ for each case) the maps induced by the inclusions sends 
$e_j \mapsto e_j$ for the nonzero generators in successive spaces.  \par  
Via these inclusions, the corresponding actions of $GL_m(\C)$ on these spaces (as explained in \cite{D3}) applied to either $I_m$ for the general or symmetric case, or $J_k$ for the skew symmetric case factor through the homogeneous spaces given in Table \ref{Coh.Compl.Lnk} to give diffeomorphisms to the complements of 
$\cD_m^{(*)}$ in each case.  The inclusions of the homogeneous spaces correspond to the inclusions of the spaces of nonsingular matrices.  Under this correspondence, the cohomology of the homogeneous spaces gives the cohomology of the complements of the spaces of $m \times m$ singular matrices $M^{(*)}_m \backslash \cD_m^{(*)}$.  Here we let $M^{(*)}_m$ denotes the space of $m \times m$ matrices of appropriate type. 
\par
Just as for Milnor fibers, we use multiplication to define a vanishing compact model for the complement. 
We let $\cP^{(*)} \subset M^{(*)}_m \backslash \cD_m^{(*)}$ denote the compact model for the complement in each of the three cases.  The action of $U_m$ in each case gives elements $A$ of the compact model to be products of elements of $U_m$ and hence 
$\| A \| = \sqrt{m}$.  Thus, $\cP^{(*)} \subset B_{\sqrt{m}}$.  Then, we can multiply the spaces of matrices by nonzero constants $a$ and for each case $a\cdot \cP^{(*)} \subset B_{a\sqrt{m}}$.  Then, for a neighborhood $B_{\gd}$ of $0$ in $M_{m}^{(*)}$, if $a\sqrt(m) < \gd$, then $a\cdot \cP^{(*)} \subset B_{\gd} \backslash \cD_m^{(*)}$.  \par Then, we define $\Phi : \cP^{(*)} \times (0, a) \to M_{m}^{(*)} \backslash \cD_m^{(*)}$ sending $\Phi(A, t) = t\cdot A$.  Then, $\Phi$ defines a vanishing compact model for the complement for each case.  \par
It remains to show that the kite map of size $\ell$ detects the corresponding exterior algebra given in Table \ref{V_0.compl.link} for the characteristic cohomology of the complement.  
We consider $\iti_m^{(*)} (\ell) : \bK_{\ell}(\C) \cap B_{\gevar} \to M_{m}^{(*)} $.  If $M_{\ell}^{(*)}$ denotes the embedding of the corresponding $\ell \times \ell$ matrices given above, then there is an $a >0 $ so that $a M_{\ell} ^{(*)} \subset \bK_{m}(\ell) \cap B_{\gevar}$.  Then, as in the proof of Theorem \ref{Thm5.1}, the composition 
$$ a (M_{\ell} ^{(*)}\backslash \cD_{\ell}^{(*)}) \,\, \subset \,\, (\bK_{m}(\ell)\backslash \cD_m^{(*)}) \cap B_{\gevar} \,\, \overset{\iti_m ^{(*)}(\ell)}{\longrightarrow} \,\, M^{(*)}_m\backslash \cD_m^{(*)} $$
 induces in cohomology an isomorphism from the exterior subalgebra given in Table \ref{V_0.compl.link} to a subalgebra of the cohomology of $a (M_{\ell}^{(*)}\backslash \cD_{\ell}^{(*)})$ (since it is diffeomorphic to $M_{\ell}^{(*)} \backslash \cD_{\ell}^{(*)}$).  As this homomorphism factors through $H^*(\C^n \backslash \cV_0; \bk)$, it is also an isomorphism onto a subalgebra of this cohomology.  This shows that $\iti_m^{(*)} (\ell)$ detects the exterior algebra, so by the Second Detection lemma in Part I 
\cite[Lemma 3.4]{D6}, the result follows for the complement. \par
Lastly, let $\widetilde{\gG}^{(*)}(f_0, \bk)$ denote the graded subspace of reduced homology obtained from the Kronecker dual $\gG^{(*)}(f_0, \bk)$ to this subalgebra.  Then, by Alexander duality we obtain a graded subspace of $H^*(L(\cV_0); \bk)$ isomorphic to $\widetilde{\gG}^{(*)}(f_0, \bk)$.  It remains to show it is obtained from the exterior algebra by truncating it and applying an appropriate shift.  As the exterior algebra satisfies Poincare duality under multiplication, this is done using the same argument in the proof of \cite[Prop. 1.9]{D3}.  
\end{proof}
\par
\begin{table}
\begin{tabular}{|l|c|l|}
\hline
Determinantal    & $\cC^{(*)}(f_0,\bk)$   & 
\,\,Shift for Link \\
  Hypersurface Type    & contains subalgebra   &    \\
\hline
$\cD_m^{sy}$  &  $\gL^*\bk\langle e_1, e_5, \dots , 
e_{2\ell-1}\rangle$  & $2n - \binom{\ell + 1}{2} - 2$ \\  
$\ell$ odd &    &    \\
\hline
$\cD_m^{sy}$  &  $\gL^*\bk\langle e_1, e_5, \dots , 
e_{2\ell-3}\rangle$  & $2n - \binom{\ell}{2} - 2$ \\  
$\ell$ even &    &    \\
\hline
$\cD_m$  & $\gL^*\bk\langle e_1, e_3, \dots , 
e_{2\ell-1}\rangle$  &  $2n -\ell^2$ - 2  \\
\hline
$\cD_m^{sk}$ (m = 2k)  & $\gL^*\bk\langle e_1, e_5, 
\dots , e_{2\ell-3}\rangle$  &  $2n - \binom{\ell}{2} - 2$  \\
$\ell$ even &    &    \\
\hline
\end{tabular}
\caption{The characteristic cohomology with coefficients in a field $\bk$  of characteristic $0$ for $\cV_0 = f_0^{-1}(\cV)$ for each matrix type $\cV = \cD_m^{(*)}$.   If $f_0$ contains an unfurled kite map of size $\ell$, the characteristic cohomology $\cC^{(*)}(f_0, \bk)$ contains an exterior subalgebra given in column 2 (where $e_j$ has degree $j$).  Then, for the link $L(\cV_0)$, the characteristic cohomology contains as a graded subspace the exterior algebra in column 2 truncated in the top degree and shifted by the degree indicated in the last column.  For the complements in the general or skew-symmetric cases, $\bk$ in column 2 may be replaced by any coefficient ring 
$R$.}
\label{V_0.compl.link}
\end{table}
\par
We reconsider the examples from \S \ref{S:sec6}
\begin{Example}
\label{Exam8.11}
In Example \ref{Exam6.1}, we considered a singularity $\cV_0$ defined by  $f_0 ; \C^9, 0 \to Sym_4(\C), 0$ given by the matrix in Figure \ref{fig:unfkitefigex1}.  It contains an unfurled kite map of size $3$.   We can apply Theorem \ref{Thm8.7}.  \par
For coefficients a field $\bk$ of characteristic $0$, from Table \ref{V_0.compl.link} the characteristic cohomology of the complement of $\cC^{(sy)}(f_0, \bk)$ contains an exterior algebra $\gL^*\bk \langle e_1, e_5 \rangle$, so there are $\bk$-vector space generators $1$, $e_1$, $e_5$, and $e_1\cdot e_5$ in degrees $0$, $1$, $5$ and $6$.  \par
The characteristic cohomology $\cB^{(sy)}(f_0, \bk)$ of the link of $\cV_0$, contains 
the subspace obtained by upper truncating the exterior algebra to obtain the $\bk$ 
vector space $\bk \langle 1, e_1, e_5 \rangle$ and shifting by 
$2\cdot 9 - 2 - \binom{4}{2} = 10$ to obtain $1$-dimensional generators in degrees $10$, $11$, and $15$.  
We note that the Link $L(\cV_0)$ has real dimension $15$, so a vector space 
generator of the characteristic subalgebra generates the top dimensional class.  \par
We also note that from Table 2 that $\cD_4^{(sy)}$ has link cohomology given by the upper truncated $\gL^*\bk \langle e_1, e_5 \rangle$ but shifted by 
$\binom{5}{2} + 4 - 2 = 12$ so there is $1$ dimensional cohomology in degrees 12, 13, and 17.  Thus, $f_0^*$ does not send any of these classes to nonzero classes in the characteristic cohomology.  \par
We do note that for the kite map $\iti_{4}^{(sy)}(3) : \C^7, 0 \to Sym_4(\C), 0$ the characteristic cohomology for the link is the upper truncated exterior algebra giving the $\bk$ vector space $\bk \langle 1, e_1, e_5 \rangle$ and then shifted by $6$.  Thus, its degrees are $6$, $7$ and $11$.  We see that as noted in \cite[Remark 1.8]{D6} in terms of the relative Gysin homomorphism, there is a shift in degrees given by twice the difference in dimension between each of the maps.  
\end{Example}
Second, we return to Example \ref{Exam6.2}. 
\begin{Example}
\label{Exam8.12}
From Example \ref{Exam6.2}, the singularity $\cV_0 = f_0^{-1}(\cD_5) $ is defined by $f_0 ; \C^{21}, 0 \to M_5(\C), 0$, given by the matrix in Figure \ref{fig:unfkitefigex2}.  Also, $f_0$ contains the linear kite map of size $4$.  Thus, we may apply Theorem \ref{Thm8.7}, the characteristic cohomology $\cC(f_0, R)$, for any coefficient ring $R$, contains the subalgebra $\gL^*R \langle e_1, e_3, e_5, e_7 \rangle$. 
Hence, characteristic cohomology of the complement has $R$ rank at least $1$ in all degrees between $0$ and $16$, except for $2$ and $14$, and it is rank at least $2$ in degree $8$.  \par 
The characteristic cohomology $\cB^{(sy)}(f_0, \bk)$ of the link contains the subspace obtained by upper truncating the exterior algebra over $\bk$ obtained from the same 
$\bk$ vector space by removing the generator of degree $16$ given by the product $e_1\cdot e_5 \cdot e_7 \cdot e_9$.  Then, we shift the resulting vector space by $2\cdot 21 - 2 - 4^2 = 24$ to obtain $1$-dimensional generators in all degrees between $24$ and $39$, except for $26$ and $38$, and it is dimension at least $2$ in degree $32$.  We note that the Link $L(\cV_0)$ has real dimension $39$, so again a vector space generator of the characteristic subalgebra generates the top dimensional class.  
\end{Example}

\section{Characteristic Cohomology for Non-square Matrix Singularities} 
\label{S:sec8a}
We extend the results for $m \times m$ general matrices and matrix singularities 
to non-square matrices. 
\subsection*{General $m \times p$ Matrix Singularities with $m \neq p$:} \hfill \par 
Let $M = M_{m, p}(\C)$ denote the space of $m \times p$ complex matrices (where we will assume $m \neq p$, with neither $= 1$).  We consider the case where 
$m > p$.  The other case $m < p$ is equivalent by taking transposes. The varieties of singular $m \times p$ complex matrices, $\cD_{m, p} \subset M_{m, p}(\C) $, with 
$m \neq p$ were not considered earlier because they do not have Milnor 
fibers.  However, the methods we applied earlier to $m \times m$ general matrices will also apply to the complement and link of $\cD_{m, p}$.  We explain that the complement has a compact homotopy model given by a Stiefel manifold.  As for the case of $m \times m$ matrices, it has a Schubert decomposition using the ordered factorization by \lq\lq pseudo-rotations\rq\rq\, due to the combined work of J. H. C. Whitehead \cite{W}, C. E. Miller \cite{Mi}, and I. Yokota \cite{Y}.  The Schubert cycles give a basis for the homology and the Kronecker dual cohomology classes which can be identified with the classes computed algebraically in \cite[Thm. 3.10]{MT} (or see e.g. \cite[\S 8]{D3}).  Thus, for appropriate coefficients, the form of both $\cC_{\cV}(f_0, R)$ and $\cB_{\cV}(f_0, \bk)$ can be given for $\cV = \cD_{m, p}$ and $f_0 : \C, 0 \to M_{m, p}(\C), 0$.  \par 
Then, we use the Schubert structure on the Stiefel manifolds to define vanishing compact models.  This allows us to define, as for the $m \times m$ case, kite subspaces and maps to detect nonvanishing characteristic cohomology of the complement and link.  
\par

\subsubsection*{Complements of the Varieties of Singular $m \times p$ 
Matrices} \hfill
\par
Let $M = M_{m, p}(\C)$ denote the space of $m \times p$ complex matrices. The 
varieties of singular $m \times p$ complex matrices, $\cD_{m, p}$, with 
$m \neq p$ were not considered earlier because they do not have Milnor 
fibers.  However, the methods do apply to the complement and link as a 
result of work of J. H. C. Whitehead \cite{W}.   We consider the case where 
$m > p$.  The other case $m < p$ is equivalent by taking transposes.  The complement to the variety $\cD_{m, p}$ of singular matrices and can be described as the ordered set of $p$ independent vectors in $\C^m$.  Then, the Gram-Schmidt procedure replaces 
them by an orthonormal set of $p$ vectors in $\C^m$.  This is the Stiefel 
variety $V_p(\C^m)$ and the Gram-Schmidt procedure provides a strong 
deformation retract of the complement $M \backslash \cV_{m, p}$  onto the 
Stiefel variety $V_p(\C^m)$.  Thus, the Stiefel variety is a compact model 
for the complement.  
\subsubsection*{Schubert Decomposition for the Stiefel Variety} \hfill
\par
The work of Whitehead \cite{W}, combined with that of C. E. Miller \cite{Mi}, and I. Yokota \cite{Y}, provides a Schubert-type cell decomposition for $V_p(\C^m)$ similar to that given in the $m \times m$ case.  There is an 
action of $GL_m(\C) \times GL_p(\C)$ on $M_{m, p}(\C)$ which is appropriate for 
considering $\cK_M$ equivalence of $m \times p$ complex matrix singularities.  However, just for understanding the topology of the link and complement of $\cD{m, p}$ it is sufficient to consider the left action of $GL_m(\C)$ acting on $M$ with an open orbit consisting of the matrices of rank $p$.  As explained in \cite{D4}, the complement $M_{m. p}(\C) \backslash \cD_{m, p}$ is diffeomorphic to the homogeneous space $GL_m(\C) /GL_{m - p}(\C)$.  The 
diffeomorphism is induced  by $GL_m(\C) \to M_{m, p}(\C)$ given by $A \mapsto A \cdot \begin{pmatrix}  I_p \\ 0_{m-p, p}\end{pmatrix}$.  Here the subgroup $GL_{m - p}(\C)$  represents the subgroup of elements $\begin{pmatrix} I_p & 0 \\ 0 & A \end{pmatrix}$ with $A \in GL_{m - p}(\C)$.  This gives the isotropy subgroup of the left action on $\begin{pmatrix}  I_p \\ 0_{m-p, p}\end{pmatrix}$.  \par
For successive values of $m$, we have the successive inclusions:  $GL_{m- 1}(\C) \hookrightarrow GL_m(\C)$ by $A \mapsto \begin{pmatrix} 1 & 0 \\ 0 & A \end{pmatrix}$.
These induce inclusions  
$$ \iti_{m-1, p- 1} : GL_{m - 1}(\C)/GL_{m - p}(\C)  \hookrightarrow GL_m(\C) GL_{m - p}(\C)\, .$$  
There is a corresponding inclusion of the spaces of matrices $M_{m - 1, p-1}(\C) \hookrightarrow M_{m, p}(\C)$ by $B \mapsto \begin{pmatrix} 1 & 0 \\ 0 & B \end{pmatrix}$.  This inclusion induces a map of the complements of the varieties of singular matrices $$ \tilde{\iti}_{m-1, p- 1} : M_{m - 1, p-1}(\C) \backslash \cD_{m-1, p-1} \hookrightarrow M_{m, p}(\C) \backslash \cD_{m, p}\, .$$ 
\par 
The actions of the groups on the spaces of matrices commute via the inclusions of the groups with the corresponding inclusions of spaces of matrices.  Thus, we have a commutative diagram of diffeomorphisms and inclusions 
\begin{equation}
\label{CD5.5a}
\begin{CD}
   {GL_{m - 1}(\C)/GL_{m - p}(\C)} @>{\iti_{m-1, p- 1}}>> {GL_m(\C) /GL_{m - p}(\C)}\\
  @V{\simeq}VV @V{\simeq}VV  \\
{M_{m - 1, p-1}(\C) \backslash \cD_{m-1, p-1}} @>{\tilde{\iti}_{m-1, p- 1}}>> {M_{m, p}(\C) \backslash \cD_{m, p}} \\
\end{CD} 
\end{equation}
\par 
The homogenenous spaces $GL_m(\C) /GL_{m - p}(\C)$ are homotopy equivalent to the homogenenous spaces given as the quotient of their maximal compact subgroups $U_m/U_{m - p}$.  Via the vertical isomorphism in \eqref{CD5.5}, the complement is diffeomorphic to the Stiefel variety $V_p(\C^m)$.  \par
By results of Whitehead \cite{W} applied in the complex case (see e.g. \cite[\S 3]{D4}), the Schubert cell decomposition of $V_p(\C^m)$ is given via ordered factorizations of matrices in $U_m$ into products of \lq\lq pseudo-rotations\rq\rq.  For this we use the reverse flag with $\tilde{e}_j = e_{m+1-j}$ for $j = 1, \dots, m$ and $\C^k$ spanned by 
$\{\tilde{e}_1, \dots , \tilde{e}_k\}$.  Then, any $B \in U_m$ can be uniquely written by a factorization in decreasing order.
\begin{equation}
\label{Eqn8.3a}
B \, \, = \,\,  A_{(\theta_k, v_k)} \cdots A_{(\theta_2, v_2)} \cdot 
A_{(\theta_1, v_1)}\, ,
\end{equation}
with 
$v_j \in_{\min} \C^{m_j}$ and $1 \leq m_1 < m_2 < \cdots < m_k \leq m$, 
and each $\theta_i \not \equiv 0  \,\mod 2\pi$.  Here $v_j \in_{\min} \C^{m_j}$ means $v_j \in \C^{m_j}$ but $v_j \not \in \C^{m_j-1}$.   Also, each $A_{(\theta_j, v_j)}$ is a pseudo-rotation about $\C<v_j>$, which is the identity on $\C<v_j>^{\perp}$ and multiplies $v_j$ by $e^{\theta_j \iti}$.  In \cite[\S 3]{D4} the results are given for increasing factorizations,; however, as explained there, the results equally well hold for decreasing factorizations.  
If $m_{k^{\prime}} > m-p \leq m_{k^{\prime} +1}$, then each $A_{(\theta_j, v_j)}$ for $j > k$ belongs to $U_{m-p}$.  Hence, $B$ is in the same $U_{m-p}$-coset as
$$ B^{\prime}  \, \, = \,\,  A_{(\theta_k, v_k)} \cdots A_{(\theta_k, v_k)}\, .$$
Then, the projections $p_{m, p}:  U_m \to U_m/U_{m-p}$ of the Schubert cells $S_{\bm}$ for $\bm = (m_1, \dots , m_k)$ with $m-p < m_1 < \dots < m_k \leq m$ give a cell 
decomposition for $U_m/U_{m-p} \simeq V_p(\C^M)$.  Furthermore, the closures 
$\overline{S_{\bm}}$, which are the Schubert cycles, are \lq\lq singular manifolds\rq\rq\ which have Borel-Moore fundamental classes (see e.g. comment after \cite[Thm. 3.7]{D4}).   

\subsubsection*{Cohomology of the Complement and Link} \hfill
\par
We can give a relation between the homology classes given by the Schubert cycles resulting from the Whitehead decomposition and  the cohomology with integer coefficients of the Stiefel variety, and hence the complement of the variety $\cD_{m, p}$ (computed in \cite[Thm. 8.10a]{MT}).  
\begin{Thm}
\label{Thm8.10a}
The homology of the complement of $\cD_{m, p}$ ($\simeq H_*(V_p(\C^m); \Z)$) 
has for a free $\Z$-basis the fundamental classes of the Schubert cycles, given as images $p_{m, p\, *}(\overline{S_{\bm}})$, with $\bm = (m_1, m_2, \dots m_k)$ for 
$m-p < m_1 < \cdots m_k \leq m$, as we vary over the Schubert decomposition of 
$U_m/U_{m-p}$.  The Kronecker duals of these classes give the $\Z$-basis for the 
cohomology, which is given as an algebra by
\begin{equation}
\label{Eqn8.1a}
 H^*(M_{m, p} \backslash \cD_{m, p}; \Z) \,\, \simeq \,\, \gL^*\Z\langle 
e_{2(m-p)+1}, e_{2(m-p)+3}, \dots , e_{2m-1}\rangle   
\end{equation}
with degree of $e_j$ equal to $j$.  \par
Moreover, the Kronecker duals of the {\em simple Schubert classes} 
$S_{(m_1)}$ for $m-p < m_1 \leq m$ are homogeneous generators of the exterior algebra cohomology. 
\end{Thm}
\begin{proof}
The computation of $H^*(V_p(\C^m)$ is given in \cite[Thm. 3.7]{D4}.  As it is a 
homotopy model for the complement \eqref{Eqn8.1a} follows.  \par
Second, that the Schubert cycles form a basis for the homology follows exactly as in the proof of \cite[Thm 6.1]{D4}, as does the proof that the Kronecker duals to the simple Schubert cycles provide homogeneous generators of the exterior algebra.   
\end{proof}
\subsubsection*{Cohomology of the Link} \hfill
\par
As a consequence of Theorem \ref{Thm8.10a}, we obtain the following conclusion for the link.
\begin{Thm}
\label{Thm8.4a}
For the variety of singular $m \times p$ complex matrices, $\cD_{m, p}$ 
(with $m > p$), the cohomology of the link is given (as a graded vector space) 
as the upper truncated  cohomology $H^*(M_{m, p} \backslash \cV_{m, p},\bk)$ 
given in \eqref{Eqn8.1a} and shifted by $p^2-2$.  \par
The Alexander duals of the Schubert cycles of nonmaximal dimension give a basis for the cohomology of the link.
\end{Thm}
\subsection*{Kite Spaces and Maps for $m \times p$ Matrix Singularities with $m \not = p$:}  \hfill 
\par
\begin{Definition}
\label{Def4.1a}  For $m \times p$ matrices with $m > p$, with $p \neq 1$ and the reverse standard flag of subspaces of $\C^m$, the corresponding {\em linear kite subspace of length $\ell$} is the linear subspace of the space of matrices defined as follows:
For $M_{m, p}(\C)$, it is the linear subspace $\bK_{m, p}(\ell)$ spanned by \par
$$  \{E_{i, j} : r + 1 \leq i \leq m, r + 1 \leq j \leq p\} \cup \{E_{i, i} : 1 \leq i \leq r\}$$
where $r = p - \ell$.  \par
Furthermore, we refer to the germ of the inclusion $\iti_ {m, p}(\ell) : \bK_ {m, p}(\ell), 0 \to M_{m, p}(\C), 0$, for each of the three cases as a {\em linear kite map of size $\ell$}.  Furthermore, a germ which is $\cK_M$ equivalent to $\iti_ {m, p}(\ell)$ will be refered to as an {\em unfurled kite map of length $\ell$}.  We also say that a germ $f_0 : \C^n, 0 \to M_{m, p}(\C), 0$ contains a kite map of length $\ell$ if there is an embedding $g : \bK_ {m, p}(\ell), 0 \to \C^n, 0$ so that $f_0 \circ g$ is an unfurled kite map of colength $\ell$.  
\end{Definition}
The general form of elements, \lq\lq the kites\rq\rq\, in the linear kite subspaces have the form given in \eqref{Eqn8.15a} 
\begin{equation}
\label{Eqn8.15a}
Q_{\ell, m - \ell} \,\, = \,\, \begin{pmatrix} D_{r}  & 0_{m - r, p - r}  \\
		0_{m - r, p} & A_{m - r, p - r} 
\end{pmatrix}
\end{equation}
where $r = p - \ell$ and $A_ {m - r, p - r} $ is an $(m - r) \times (p - r)$-matrix which denotes an arbitrary matrix in $M_{m - r, p - r}(\C)$.  Also, $0_{q, s}$ denotes a $0$-matrix of size $q \times s$ and $D_{r}$, an arbitrary $r \times r$ diagonal matrix.  
The general element is exhibited in Figure \ref{fig:altkitefiga}.  
\begin{Remark}
\label{Rem8.15a}
Although the body of the kite is not square, the length $\ell$ denotes the rank of a generic matrix in the body, which is consistent with the square case when $m = p$.  We note that to be consistent with the form of the matrices for the group representation of the complement and the Schubert decomposition for the nonsquare case, the kite is 
\lq\lq upside down\rq\rq.  However, elements of $\cK_M$ allow for the composition with invertible matrices $GL_m$ and $GL_p$ with entries in the local ring of germs.  This allows for a linear change of coordinates so the kite can be inverted to the expected form as for the $m \times m$ case.  
\end{Remark}
\par 
\begin{figure}
$$ \begin{pmatrix} * & \cdots & 0  & 0  & \cdots & 0 \\
	0 & \cdots & 0 & 0  & \cdots & 0 \\
0 & \cdots & * & 0  & \cdots & 0 \\
0 & \cdots & 0 & *  & \cdots &* \\
0 & \cdots & 0 & * & \ddots  & * \\
0 & \cdots & 0 & * & \cdots  & *
\end{pmatrix}
$$
\caption{Illustrating the form of elements of a linear kite space of length $\ell$ in the space of general $m \times p$ matrices with $r = p - \ell$. The upper $r \times r$ left matrix is a diagonal matrix with arbitrary entries, and the lower right matrix is a general matrix of size $(m -r) \times (p - r)$.}
\label{fig:altkitefiga}
\end{figure}
\par
We have an analogue of the detection result for case of $m \times m$ general matrices. 
\begin{Thm}
\label{Thm8.7a}
Let $f_0 : \C^n, 0 \to M_{m, p}(\C), 0$ define a matrix singularity.  If $f_0$ contains a kite map of length $\ell$, then the characteristic cohomology of the complement 
$\cC_{m, p}(f_0, \bk)$, for a field $\bk$ of characteristic $0$, contains the exterior algebra given by
\begin{equation}
\label{Eqn8.7a}
  \gL^*\bk\langle e_{2(m-p)+1}, e_{2(m-p)+3}, \dots , e_{2(m - p) +2 \ell - 1}\rangle  
\end{equation}
and each $e_j$ has degree $j$.  \par 
Furthermore, the characteristic cohomology of the link $\cB_{m, p}(f_0, \bk)$, as a 
graded vector space contains the graded subspace given by truncating the top degree of the exterior subalgebra \eqref{Eqn8.7a} of $\cC_{m, p}(f_0, \bk)$ and shifting by $2n - 2 - \ell \cdot (2(m - p) + \ell)$.  
\par 
For the complement, $\bk$ may be replaced by any coefficient ring $R$.  
\end{Thm}
\begin{proof} \par 
\par
The line of proof follows that for the $m \times m$ general case.  \par
Under the inclusion $\iti_{m-1,p-1}: V_{p-1}(\C^{m-1}) \hookrightarrow V_{p}(\C^{m})$, the 
identification of the cohomology classes with Kronecker duals of the Schubert cycles implies 
$$ \iti_{m-1,p-1}^*(e_{2(m-p+ j) -1}) = e_{2(m-p+ j) -1} \, \text{for}\,  1 \leq j \leq p-1 \, \text{and}\, \iti_{m-1,p-1}^*(e_{2m -1}) = 0 \, .$$  
If we compose successive inclusions $\ell$ times to give  $\iti_{m-\ell,p-\ell}: V_{p-\ell}(\C^{m-\ell}) \hookrightarrow V_{p}(\C^{m})$, then the induced map on cohomology has image the algebra given in \eqref{Eqn8.7a}.  Thus, $V_{p}(\C^{m})$ provides a compact model for the complement, and the composition $\iti_{m - 1, p - 1} \circ \iti_{m - 2, p - 2} \circ \cdots \circ \iti_{m - r, p-r}$ with $r = p - \ell$ detects the subalgebra in 
\eqref{Eqn8.7a}.   
\par
Now using the vanishing compact model $t\cdot V_{p}(\C^{m})$, we can follow the same reasoning as for the $m \times m$ case using the functoriality and invariance under $\cK_{\cD_{m, p}}$, and apply the Second Detection Lemma to obtain the result. \par
Then, as the exterior algebra satisfies Poincare duality under multiplication, we can deduce the result for $\cB_{m, p}(f_0, \bk)$ using the same argument in the proof of \cite[Prop. 1.9]{D3} where for the shift $2n - 2 - \dim_{\R} K$ we replace $\dim_{\R} K$ by the top degree of the algebra in \eqref{Eqn8.7a}.  This is the same as 
$\dim_{\R} V_{p - r}(\C^{m - r})$, which is 
$$ 2(p - r)((m - r) -(p - r)) + (p-r)^2 \,\, = \,\, (p-r)(2(m - p) - r) \,\, = \,\, \ell (2(m - p) + \ell) \, .$$
\end{proof}
\par
\begin{Example}
\label{Ex8.1a}
Consider an example of a matrix singularity which is given by $f_0 : \C^{12}, 0 \to M_{4, 5}(\C), 0$ defined by the matrix in Figure \ref{fig:unfkitefigex4} for which all $g_{i, j}(\bx,0) = 0$.
\begin{figure}[h]
$$ 
\begin{pmatrix} 
x_{1,1}  & x_{1,2}   & x_{1,3} & g_{1, 4}(\bx, \by)  & g_{1, 5}(\bx, \by)   \\
x_{2,1}  & x_{2,2} & x_{2,3} & g_{2,4}(\bx, \by) &  g_{2, 5}(\bx, \by) \\
g_{3, 1}(\bx, \by) & g_{3, 2}(\bx, \by)  & y_3  & x_{3,4}  & y_4 \\
y_1 & y_2 & g_{4, 3}(\bx, \by) & g_{4,4}(\bx, \by) &  x_{4,5}  
\end{pmatrix}
$$
\caption{An example of a $4 \times 5$ matrix singularity $f_0$, with $g_{i, j}(\bx, 0) \equiv 0$ for each $(i, j)$.  It contains a kite map of colength $2$ given when all $y_i = 0$.}
\label{fig:unfkitefigex4}
\end{figure}
For $\by = (y_1, y_2, y_3, y_4)$, when $\by = 0$, we see that $f_0$ contains a kite map of colength $2$.  Then, Theorem \ref{Thm8.7a} implies that $\cC_{\cD_{4, 5}}(f_0, \Z)$ contains a subalgebra $\gL^*\Z\langle e_3, e_5 \rangle$. 
Also, by Theorem \ref{Thm8.7a}, $\cB_{\cD_{4, 5}}(f_0, \Z)$ contains as a subgroup the subalgebra upper truncated and then shifted by 
$2\cdot 12 - 2 - (4 - 2)(2(5 - 4) + 2) = 14$.  Thus, the classes 
$\{ 1, e_3, e_5\}$ are shifted by $14$ to give classes in degrees 
$14, 17, 19$.  As $\cV_0 = f_0^{-1}(\cD_{4, 5})$ has codimension $2$, the link 
$L(\cV_0)$ has dimension $19$ and the characteristic cohomology class in degree $19$ generates the Kronecker dual to the fundamental class of $L(\cV_0)$.  
\end{Example}
\par
\section{Cohomological Relations between Local Links via Restricted Kite Spaces}
\label{S:sec8b}
\par 
Lastly, it is still not well understood how the structure of the strata for the varieties of singular matrices contributes to the (co)homology of the links for the various types of matrices.  We use kite spaces for all of the cases to determine the relation between the cohomology of local links for strata with the cohomology of the global link.  This includes as well the relation between the local links for strata with local links of strata of higher codimension.  This is via the relative Gysin homomorphism defined as 
\cite[(1.10)]{D6}, which is an analog of the Thom isomorphism theorem in these cases.  
\par
We do so by explaining how the kite subspaces provide transverse sections to the strata of the varieties of singular matrices for all three cases of $m \times m$ matrices and also for general $m \times p$ matrices.  To consider all cases simultaneously, we denote the corresponding space of matrices by $M$ and the variety of singular matrices by $\cD_{*}^{(*)}$.  Also, we consider the kite subspace of length $\ell$ of appropriate type which we denote by $\bK_{*}^{*}(\ell)$.  For the $m \times m$ cases, we also let $r = m - \ell$ (which is the same as $p - \ell$ when $m = p$).  
\par 
We consider an affine subspace obtained by choosing fixed nonzero values at the 
entries in the tail (e.g. the value $1$).  When the entries in the body of the kite are $0$, we obtain a matrix $A$ of rank $r$ and hence corank $\ell$.  Then, the resulting space we consider has the form $A + M^{\prime}$ where $M^{\prime}$ denotes one of the spaces $M_{\ell}(\C)$, $M_{\ell}^{(sy)}(\C)$, $M_{\ell}^{(sk)}(\C)$, or 
$M_{m - r, p - r}(\C)$ which is embedded, via a map denoted by $\iti$, as the body of the kite.  
This provides a normal section to the stratum $\gS_{\ell}$ of matrices of corank $\ell$ through $A$.  We refer to this affine subspace as a {\em restricted kite space}.  We let 
$\cD_{*}^{(*) \prime}$ denote the variety of singular matrices in $M^{\prime}$.  Then, in 
a sufficiently small tubular neighborhood $T$ of $\gS_{\ell}$ we obtain $\cD_{*}^{(*)} \cap T$ is diffeomorphic to $\gS_{\ell} \times (\cD_{*}^{(*) \prime} \cap B_{\gevar})$ for sufficiently small 
$\gevar > 0$.  We refer to $L(\cD_{*}^{(*) \prime})$ as the {\em local link of the stratum $\gS_{\ell}$}.  \par 
Then $\iti$ induces an inclusion $\iti : \cD_{*}^{(*) \prime} \cap B_{\gevar} \subset \cD_{*}^{(*)}$.  There is the induced map $\iti^*$ on cohomology which sends the exterior algebra giving the cohomology of $M \backslash \cD_{*}^{(*)}$ to the algebra \eqref{Eqn8.7a}.  This is a consequence of the proofs of Theorems \ref{Thm8.7}
 and \ref{Thm8.7a}.  Using this we have consistent monomial bases for the cohomology of the complement.  This allows us to define consistent Kronecker pairings giving a 
well-defined relative Gysin homomorphism (as defined in \cite[(1.10)]{D6}).  There is the following relation between the cohomology of the local link $L(\cD_{*}^{(*) \prime})$ and the link $L(\cD_{*}^{(*)})$.
\begin{Corollary}
\label{Cor8.1a}
The relative Gysin homomorphism 
$$\iti_* : H^*(L(\cD_{*}^{(*) \prime}); \bk) \to H^{*+ q}(L(\cD_{*}^{(*)}); \bk)$$
 increases degree by $q = \dim_{\R} M - \dim_{\R} M^{\prime}$, which in the various cases equals for the $m \times m$ cases: $2(m^2 - \ell^2)$ for the general matrices; 
$(m - \ell)(m + \ell +1)$ for symmetric matrices,  $(m - \ell)(m + \ell -1)$ for skew-symmetric matrices (with $m$ and $\ell$ even); and for $m \times p$ matrices 
$2(p^2 - \ell^2)$.  \par 
It is injective and sends the Alexander dual of the Kronecker dual of a class corresponding to a monomial in the algebra \eqref{Eqn8.7a} to the corresponding Alexander dual of the Kronecker dual of the image of that class as an element of the cohomology of the complement $M \backslash \cD_{*}^{(*)}$.
\end{Corollary}
\begin{proof}  By the above remarks, there is defined the relative Gysin homomorphism.  If $\iti$ denotes the inclusion of the reduced kite space into the space of matrices, then the induced map on cohomology of the complements, denoted $\iti^*$, (with coefficients $\bk$ a field of characteristic $0$) is surjective.  We use the identification of the monomials $e_{\bm}$ with the Kronecker duals denoted $e_{\bm}^*$.  Then, the inclusion  $\iti_*$ is the dual of $\iti^*$.  Thus, the dual homomorphism for homology $\iti_*$ is injective.  When this is composed with Alexander isomorphisms (via the Kronecker pairings), it remains injective.  
By the properties of the corresponding cohomology classes of the links resulting from applying Alexander duality have the effect of raising degree by the difference $\dim_{\R} M - \dim_{\R} M^{\prime}$ for each of the four types.  These are then computed to give the stated degree shifts.  
\end{proof}
\par
We also mention that there is an analogous version of this corollary for the case of the local link for a stratum $\gS_{\ell^{\prime}}$ included in the local link of a stratum $\gS_{\ell}$ for $\ell^{\prime} < \ell$.   As an example we consider
\begin{Example}
\label{Ex8.2a}
For the stratum $\gS_2 \subset Sym_5(\C)$, the local link has reduced cohomology 
group isomorphic to $\widetilde {\gL*\bk}\langle e_1, e_5\rangle[4]$.  However, the effect of Alexander duality on elements does not correspond to a shift.  The reduced cohomology of the local link complement is spanned by the generators $\{ e_1, e_5, e_1e_5\}$ with Kronecker duals denoted $\{ e_{1\, *}, e_{5\, *}, (e_1e_{5})_*\}$.  Then, the corresponding Alexander dual generators for the reduced cohomology of the local link, denoted $\{\widetilde{e_1}, \widetilde{e_5}, \widetilde{e_1e_5}\}$, have degrees in cohomology $9, 5, 4$ in that order.  Note under the shift representation 
$\widetilde{e_1e_5}$ corresponds to the shift of $1$.  \par 
Also, the link $L(\cD_5^{(sy)})$ has cohomology group $\widetilde {\gL*\bk}\langle e_1, e_5, e_9\rangle[13]$.  Then, as 
$\iti^*$ is surjective, $\iti_*$ maps the Kronecker duals $\{e_{1\, *}, e_{5\, *}, (e_1e_{5})_*\}$ to homology classes of the same degrees for the complement $Sym_5(\C) \backslash \cD_5^{(sy)}$.  
Then, these elements correspond under Alexander duality to cohomology classes $\{\widetilde{e_1^{\prime}}, \widetilde{e_5^{\prime}}, \widetilde{(e_1e_5)^{\prime}}\}$ for the link $L(\cD_5^{(sy)})$, having degrees $27, 23, 22$.  We see that the increase in degree is $2 (15 -6) = 18$ as asserted for the relative Gysin homomorphism..  
\par 
However, one key point to note is that for the cohomology group of the link represented by the truncated and shifted exterior algebras, the relative Gysin homomorphism does not map the shifted classes to the corresponding shifted classes.  For example, for the local link of $\gS_2 \subset Sym_5(\C)$, $e_1$ corresponds to $\tilde{e_1}$ which maps to a cohomology class of degree $27$, while for the link $L(\cD_5^{(sy)})$ , $e_1$ corresponds via the shift representation to a cohomology class in the link of degree $14$. 
\end{Example}
\par
\begin{Remark} 
\label{Rem8.12}
For $m \times p$ finitely $\cK_M$-determined matrix singularities $f_0 : \C^n, 0 \to M_{m, p}(\C), 0$, if $n < | 2 (m-p +2)|$, then by transversality, $\cV_0$ has an isolated singularity and so a stabilization provides a Milnor fiber as a particular smoothing.  As yet there does not appear to be a mechanism for showing how this Milnor fiber inherits topology from 
$M_{m, p}(\C) \backslash \cD_{m, p}$.  However, for $(m, p) = (3, 2)$, Fr\"{u}hbis-Kr\"{u}ger and Zach \cite{F}, \cite{Z}, \cite{FZ} have shown that for the resulting Cohen-Macaulay $3$-fold singularities in $\C^5$, the Milnor fiber has Betti number $b_2 = 1$, allowing the formula of Damon-Pike \cite[\S 8]{DP} to yield an algebraic formula for $b_3$.  It remains to be understood how this extends to larger $(m, p)$. 
\end{Remark}

\end{document}